\documentclass[10pt,a4paper]{article}

\usepackage[utf8]{inputenc}
\usepackage{amsmath}
\usepackage{amsfonts}
\usepackage{amssymb}
\usepackage{amsthm}
\usepackage{mathtools}
\usepackage{mathrsfs}
\usepackage{hyperref}
\usepackage{cite}
\usepackage{bbm}
\usepackage{enumitem}
\usepackage{xypic}

\newtheorem{thm}{Theorem}
\newtheorem{lem}[thm]{Lemma}
\newtheorem{prop}[thm]{Proposition}
\newtheorem{coro}[thm]{Corollary}
\theoremstyle{definition}
\newtheorem{defn}[thm]{Definition}

\newtheorem{example}[thm]{Example}
\theoremstyle{remark}
\newtheorem{rmk}[thm]{Remark}

\DeclareMathOperator{\supp}{supp}

\DeclareMathOperator{\dist}{dist}

\DeclareMathOperator{\vol}{vol}

\newcommand{\R}{\mathbb{R}}
\newcommand{\Z}{\mathbb{Z}}

\newcommand{\Q}{\mathbb{Q}}

\newcommand{\N}{\mathbb{N}}

\newcommand{\lebesgue}{\mathscr L}

\newcommand{\T}{\mathbb T}
\newcommand{\magicK}{\bar C}
\renewcommand{\geq}{\geqslant}
\renewcommand{\leq}{\leqslant}

\newcommand{\kam}{{\textsc{kam}}}
\newcommand{\kamt}{KAM}

\DeclareMathOperator{\Hom}{Hom}
\newcommand{\Homt}[1]{\Hom^{#1}}
\date{}
\title{Weak KAM theory in higher-dimensional holonomic measure flows}
\author{Rodolfo R\'ios-Zertuche
}

\begin{document}

\maketitle
\begin{abstract}
We construct a weak \textsc{kam} theory for parameterized cobordisms and their relaxation, holonomic measures. %
We find a weak \textsc{kam} solution in that context, and we show that in many cases it corresponds to an exact form that satisfies a version of the Hamilton-Jacobi equation. Along the way, we give a characterization of minimizable Lagrangians, as well as some abstract weak \textsc{kam} machinery. \end{abstract}

\section{Introduction}
\label{sec:intro}

In this paper we construct a weak \kam{} theory for Lagrangian action functionals on a relaxation of the space of parameterized cobordisms. Let us briefly explain how this works. Given a manifold $M$ of dimension $d$ and two submanifolds $N_1$ and $N_2$ of dimension $n-1$, a cobordism joining $N_1$ and $N_2$ is a submanifold $B$ of dimension $n$ whose boundary is exactly the two submanifolds $N_1$ and $N_2$. Let us say that the manifold $B$ can be parameterized by a set of maps $\varphi_i\colon U_i\subseteq\R^n\to M$, $i=1,\dots, k$, so that the derivative $d\varphi_i$ is a map from some open subset $U_i\subseteq\R^n$ to $T^nM=TM\oplus\dots\oplus TM$, the Whitney sum of $n$ copies of $TM$ in such a way that $\{\varphi_i(U_i)\}_{i=1}^k$ is a partition of $M$ up to a set of measure zero. Then the action of $B$ with respect to the Lagrangian function $L\colon T^nM\to\R$ is defined to be
\[\sum_i\int_{U_i}L(d\varphi_i(x))\,d\lebesgue_{U_i}(x),\]
where $\lebesgue_U$ denotes the Lebesgue measure on $U$. This integral is, in general, dependent on the parameterization $\{\varphi_i\}_i$ of $B$. An example of such a functional is the squared ssurface area, which corresponds to setting
 \[L(x,v_1,\dots,v_n)=\vol(x,v_1,\dots,v_n)^2=|\det (g_x(v_i,v_j))_{i,j=1}^n|.\] 
 Another way to think about the action of the Lagrangian function $L$ is as the integral of $L$ over the measure
\[\mu_B=\sum_i(d\varphi_i)_*\lebesgue_{U_i}\]
on $T^nM$. One can also integrate differential forms with respect to $\mu_B$, since they define real-valued functions on $T^nM$; thus $\mu_B$ induces a current $T_{\mu_B}\colon \Omega^n(M)\to\R$ with boundary equal to the difference of the currents similarly induced for the submanifolds $T_{\mu_{N_1}}$ and $T_{\mu_{N_2}}$,
$\partial T_{\mu_B}=T_{\mu_{N_2}}-T_{\mu_{N_1}}$. Note that, up to a sign, $T_{\mu_{X}}$ is independent of the parameterization of the corresponding manifold $X$. Thus we consider the following relaxation: we replace the space $\mu_B$ of measures corresponding to parameterized cobordisms by the space of measures $\mu$ on $T^nM$ such that their induced current $T_\mu$ has the right boundary, $\partial T_{\mu}=T_{\mu_{N_2}}-T_{\mu_{N_1}}$; we term these measures \emph{holonomic measures}. 

The weak \kam{} theory was originally discovered \cite{fathibook} in a very different context, namely, in the study of geodesics of Lagrangian and Hamiltonian dynamical systems. In a sense, that would correspond to the $n=1$ case in our description above: instead of cobordisms, one simply considers curves joining points on a manifold $M$. The theory then gives a function $u\colon M\to\R$ that is a viscosity solution of the Hamilton--Jacobi equation, $H(du)=c$, where $H$ is the Hamiltonian and $c\in\R$. This gives a precise description of the asymptotic dynamics of many of the geodesics of the system. A relaxation similar to the one described above was also considered in that context (see for example \cite{matheractionminimizing91,contrerasiturriagabook}). 

The weak \kam{} theory we construct in this paper has many similarities and many differences with the one obtained in the original context. Among the differences, we can note that the function $u$ produced by the main theorem is no longer a function on the manifold $M$, but on a certain space of objects that can be understood as slices of holonomic measures, analogous to the measures $\mu_{N_j}$ above: just like we can consider a cylinder as a set of circles glued together, and the cylinder then as a curve in the space of circles, a curve in the space of \emph{argute slices} will give a holonomic measure, thus giving meaning to the notion of flow in the space of slices of holonomic measures. As was the case in the original context, the theory we describe can be interpreted as giving a description of the dynamics of the geodesic flow induced by the Lagrangian action, now in the space of argute slices. Instead of geodesics, we have action-minimizing holonomic measures, which are akin to action-minimizing cobordisms. The function $u$, which will be, in an appropriate sense, a weak \kam{} solution, will also in certain circumstances correspond to a differential form $\omega$ on $M$ that will satisfy a sort of Hamilton--Jacobi equation, $H(d\omega)=c$.

The organization of the paper is as follows. Section \ref{sec:whatis} develops a point of view of what a weak \kam{} theory is, and it introduces the concepts of Lagrangian category, the Lax-Oleinik semigroup, a weak \kam{} solution, and a finely \kam{}-amenable category; it then gives a few examples from the literature, before turning to the proof a general weak \kam{} theorem with mild assumptions on the class of Lagrangian categories it applies to. Section \ref{sec:slices} is devoted to the particular example of a Lagrangian category that we are interested in, in which the objects are $(n-1)$-dimensional currents that arise as time slices of holonomic measures; in this section we give sufficient conditions for such a Lagrangian category to be weak \kam{}-amenable and we show that in that case there are weak \kam{} solutions.
Section \ref{sec:coarsechar} gives the characterization of minimizable Lagrangian actions and the connection of weak \kam{} solutions on the Lagrangian category of currents to exact forms on a manifold, and explains in which sense they satisfy a Hamilton--Jacobi equation.
\vspace{1,5mm}

\section{What is a weak \kamt{} theory?}
\label{sec:whatis}
\subsection{Abstract formulation and examples}
\label{sec:weakkam}
\newcommand{\categ}{{\mathcal C}}

In this section, as a sort of preface, we give a point of view of what constitutes a weak \kam{} theory, for which we do an abstraction of some of the material in \cite{madernafathi}, and we relate it to the original construction of Fathi--Siconolfi \cite{fathibook}.

We use the language of categories, as it appears to fit naturally and to appropriately generalize the weak \kam{} theories that we know of. However, the reader unfamiliar with categories can do the following substitutions without losing much: instead of the collection of objects of the category, think of a topological space, and instead of the morphisms of the category, think of paths joining the points of the topological space.

\subsubsection{Lagrangian categories}
\label{sec:abstractsetting}

Recall that a category is a class of objects together with a class of associative morphisms between the objects, and such that every object is endowed with an identity morphism. For objects $x$ and $y$, the class of morphisms having domain $x$ and codomain $y$ is denoted $\Hom(x,y)$. We will slightly abuse notations denoting the class of objects of a category $\categ$ also by $\categ$, so that $x\in \categ$ means that $x$ is an object of $\categ$.

A category $\categ$ with morphisms 
\[\Hom(\categ)=\bigcup_{x,y\in\categ}\Hom(x,y)\]
is \emph{pre-Lagrangian} if it is
endowed with two functions 
\[A\colon \Hom(\categ)\to \R\quad\textrm{and}\quad\mathbf{t}\colon \Hom(\categ) \to [0,+\infty) \]
associating to each morphism $\gamma\in\Hom(x,y)$ the \emph{action} $A(\gamma)$ and the \emph{internal time duration} $\mathbf t(\gamma)$ satisfying the relations
\[A(\gamma\circ\eta)=A(\gamma)+A(\eta)\quad \textrm{and}\quad \mathbf t(\gamma\circ\eta)=\mathbf t(\gamma)+\mathbf t(\eta)\]
for all objects $x,y,z\in\categ$ and all morphisms $\eta\in \Hom(x,y)$ and $\gamma\in \Hom(y,z)$ between them,  as well as 
\[A(\operatorname{id}_x)=0=\mathbf t(\operatorname{id}_x)\] 
for the identity morphism $\operatorname{id}_x\in \Hom(x,x)$.

It will be useful to denote, for $t\geq 0$, $x,y\in\categ$, by $\Homt{t}(x,y)$ the class of morphisms $\gamma\in\Hom(x,y)$ with $\mathbf t(\gamma)=t$.

Given a pre-Lagrangian category $\categ$, we define, for $t\in[0,+\infty)$ and $x,y\in\categ$, the \emph{finite-time action potential} $h_t\colon \categ\times\categ\to\R\cup\{\pm\infty\}$ by
\[h_t(x,y)=\inf_{{\gamma\in\Homt{t}(x,y)}}A(\gamma).\]
Observe that if the set of morphisms $\gamma\in\Hom(x,y)$ with $\mathbf t(\gamma)=t$ is empty, then $h_t(x,y)=+\infty$.

If the finite-time action potential satisfies, for all $s,t\geq 0$, and $x,z\in\categ$,
\begin{equation}\label{eq:lagrangianness}
 h_{s+t}(x,z)=\inf_{y\in\categ}h_s(x,y)+h_t(y,z),
\end{equation}
then we say that $\categ$ is \emph{Lagrangian}. The following result gives sufficient conditions for $\categ$ to be Lagrangian that can be checked easily in many contexts.

\begin{lem}\label{lem:sufflagrangian}
    Assume that the category $\categ$ is pre-Lagrangian and satisfies the following:
    \begin{enumerate}[label=L\arabic*.,ref=L\arabic*]
        \item\label{A:curves} For every pair of objects $x$ and $y$ in $\categ$ and every $t\geq 0$ there is a morphism $\gamma\in\Hom(x,y)$ between them with $\mathbf t(\gamma)=t$. In other words, $\Homt{t}(x,y)$ is not empty.

        \item\label{A:division} For each $\gamma\in\Hom(x,z)$ 
        and for each $t\in[0,\mathbf t(\gamma)]$, there exist an object $y$ and morphisms $\xi\in\Homt{t}(x,y)$ and $\eta\in\Homt{\mathbf t(\gamma)-t}(y,z)$ such that $\gamma=\xi\circ\eta$. 
    \end{enumerate}
    Then $\categ$ is Lagrangian.
\end{lem}
\begin{proof}
 We need to prove that \eqref{eq:lagrangianness} holds for $\categ$.
 Let $x,z\in\categ$ and fix $s,t\geq0$.
 The inequality $h_{s+t}(x,z)\leq \inf_{y\in \categ}h_s(x,y)+h_t(y,z)$ is always true. Let us verify the opposite inequality.  Let $\varepsilon>0$ and pick some $\gamma\in \Homt{s+t}(x,z)$ with %
 $A(\gamma)\leq h_{s+t}(x,z)+\varepsilon$; such $\gamma$ exists because by \ref{A:curves} the class $\Homt{s+t}(x,y)$ is not empty, and $h_{s+t}(x,z)$ is the infimum of the action $A$ on that class. Let $y_0\in\categ$, $\eta\in \Hom^s(x,y_0)$ and $\xi\in\Hom^t(y_0,z)$ be as in \ref{A:division}, so that %
 $\gamma=\xi\circ \eta$. We have
 \begin{multline*}
  \inf_{y\in\categ}h_s(x,y)+h_t(y,z)
  \leq h_s(x,y_0)+h_t(y_0,z)\\
  \leq A(\eta)+A(\xi)=A(\xi\circ\eta)=A(\gamma)\leq h_{s+t}(x,z)+\varepsilon.
 \end{multline*}
 Since we can do this for all $\varepsilon>0$, we conclude that \eqref{eq:lagrangianness} holds.
\end{proof}

Associated to the maps $h_t$, we have a set of maps $\varphi_t$ that operate on functions $u\colon \categ\to\R$ by
\[\varphi_tu(x)=\inf_{\substack{y\in \categ%
}}u(y)+h_t(y,x),\quad t\geq 0.\]
The set $\{\varphi_t:t\geq0\}$ is known as the \emph{Lax-Oleinik semigroup} because of the following result.

\begin{lem}\label{lem:semigroup}
If the category $\categ$ is Lagrangian, then
\[\varphi_s\circ\varphi_t=\varphi_{s+t},\quad s,t\geq 0.\]
\end{lem}
\begin{proof}
We observe that, as a consequence of \eqref{eq:lagrangianness},
\begin{align*}
\varphi_{s+t}u(x)
&=\inf_{y\in \categ}u(y)+h_{s+t}(y,x)\\
&=\inf_{y\in \categ}\left[u(y)+\inf_{z\in \categ} h_{t}(y,z)+h_s(z,x)\right]\\
&=\inf_{y\in \categ}\inf_{z\in \categ}[u(y)+h_{t}(y,z)+h_s(z,x)]\\
&=\inf_{z\in \categ}\inf_{y\in \categ}[u(y)+h_{t}(y,z)+h_s(z,x)]\\
&=\inf_{z\in \categ}\left[\inf_{y\in \categ}[u(y)+h_{t}(y,z)]+h_s(z,x)\right]\\
&=\inf_{z\in \categ} \varphi_{t}u(z)+h_s(z,x)=\varphi_s[\varphi_{t}u](x).\qedhere
\end{align*}
\end{proof}

The central result of a typical weak \kam{} theory, a version of which is presented below in Theorem \ref{thm:mainlemma}, is that, under certain technical assumptions on a Lagrangian category $\categ$, there exists a function $u\colon\categ\to\R$
satisfying, for some $c_0\in\R$ and for all $t\geq 0$,
\begin{equation}\label{eq:fixedpoint}
 u=\varphi_tu+c_0t.
\end{equation}
The significance of this result depends on the context of the specific category $\categ$. The reader will be well-served with  some examples, and we give some in Section \ref{sec:examples}.

\subsubsection{Weak \kamt{} solutions and finely \kamt{}-amenable categories}\label{sec:weakkamsolutions}
Let us give a useful alternative characterization of a function $u$ satisfying \eqref{eq:fixedpoint} that applies to many cases of interest. In particular, it applies to most of the examples of Section \ref{sec:examples}. 

A function $u\colon\categ\to\R$ is a \emph{weak \kam{} solution (of negative type\footnote{The original theory \cite{fathibook} includes the concept of \emph{weak \kam{} solutions of positive type} obtained by appropriately changing the signs of time parameters $s,t$, and it subsequently considers the idea of \emph{conjugate solutions}, which are both very interesting. For simplicity we do not discuss those topics here.})} if the two following conditions are verified:
\begin{itemize}
    \item There is some $c_0\in\R$ such that, for all objects $x$ and $y$ of $\categ$ and every $\gamma\in \Hom(y,x)$ we have
    \begin{equation}\label{eq:dominated}
        u(x)-u(y)\leq A(\gamma)+c_0\mathbf t(\gamma).
    \end{equation}
    
    \item For every object $x$ in $\categ$, there is a subcategory $\Gamma_x$ of $\categ$ such that:
    \begin{itemize}
    \item $x$ is an object of $\Gamma_x$, 
    \item $\Gamma_x$ contains an object $y_t$ for each $t\leq 0$ with $x=y_0$, the class of morphisms $\Hom_{\Gamma_x}(y_s,y_t)$ is nonempty for all $s\leq t\leq 0$, and 
    \item every $\gamma\in \Hom_{\Gamma_x}(y_s,y_t)$ satisfies
        \begin{equation}\label{eq:calibrated}
         \mathbf t(\gamma)=t-s\quad \textrm{and} \quad u(y_t)-u(y_s)=A(\gamma)+c_0(t-s). 
        \end{equation}
    \end{itemize}
\end{itemize}
    Analogously to the classical case, we refer to the subcategory $\Gamma_x$ as a \emph{geodesic} of the Lagrangian category $\categ$.

A category $\categ$ is \emph{finely \kam{}-amenable} if it satisfies the following:
\begin{enumerate}[ref=A\arabic*,label=A\arabic*.]
\item \label{A:lagrangian} The category $\categ$ is Lagrangian.
\item\label{A:realizable} 
(Morphisms realizable as curves) 
There is, associated to each morphism $\gamma\in\Hom(x,y)$,  
a curve $\bar\gamma\colon[0,\mathbf t(\gamma)]\to\categ$ such that $\bar \gamma(0)=x$, $\bar\gamma(\mathbf t(\gamma))=y$.

\item\label{A:composition} (Composition of morphisms is concatenation of curves) The association $\gamma\mapsto\bar\gamma$ is such that, for every two morphisms $\gamma\in \Hom(x,y)$ and $\eta\in\Hom(y,z)$, the curve $\overline{\eta\circ\gamma}$ associated to the composition $\eta\circ\gamma\in\Hom(x,z)$ is the concatenation $\bar\gamma*\bar\eta$ of $\bar\gamma%
$ followed by $\bar\eta%
$, that is,
\[\overline{\eta\circ\gamma}=\bar\gamma*\bar\eta(s)\coloneqq\begin{cases}
                          \bar\gamma(s),&s\in[0,\mathbf t(\gamma)],\\
                          \bar\eta(s-\mathbf t(\gamma)),&s\in[\mathbf t(\gamma),\mathbf t(\gamma)+\mathbf t(\eta)].
                         \end{cases}
\]

\item \label{A:restriction} 
(Restrictions of curves are morphisms)
Given $\gamma\in\Hom(\categ)$ and $0\leq \sigma\leq \tau\leq \mathbf t(\gamma)$, there is $\eta\in\Hom(\bar\gamma(\sigma),\bar\gamma(\tau))$ such that $\bar\eta(t)=\bar\gamma(t+\sigma)$, $t\in[0,\tau-\sigma]$, and $\mathbf t(\eta)=\tau-\sigma$. (We will slightly abuse notations to denote such $\eta$ by $\gamma|_{[\sigma,\tau]}$.)

\item\label{A:minimum} 
(Attainment of a certain minimum)
There is $k_0>0$ such that, every $x\in\categ$ and %
each function $u\colon\categ\to\R$ satisfying \eqref{eq:fixedpoint}, there is $t\geq k_0$ such that the function 
\begin{equation*}%
 (y,\gamma)\mapsto u(y)+A(\gamma)%
\end{equation*}
 attains its minimum in the class of pairs $(y,\gamma)$ with $\gamma\in\Hom(y,x)$ with $\mathbf t(\gamma)\geq k_0$.%
\end{enumerate}
Our main motivation for requiring \ref{A:minimum} is that it provides the central tool to prove Lemma \ref{lem:amenablessolvable}.

\begin{lem}\label{lem:amenablessolvable}
On a finely \kam{}-amenable  category $\categ$, a function $u$ satisfies \eqref{eq:fixedpoint} if, and only if, $u$ is a weak \kam{} solution.
\end{lem}

\begin{proof}
Let us first assume that $u$ satisfies \eqref{eq:fixedpoint}, which we may rewrite as
\[u(x)=\inf_{\substack{y\in\categ\\\gamma\in\Homt{t}(y,x)}}u(y)+A(\gamma)+c_0\mathbf t(\gamma).\] 
This clearly implies that \eqref{eq:dominated} holds.

In order to show that $u$ is a weak \kam{} solution,
let $x$ be an object in $\categ$ and let us show that the subcategory $\Gamma_x$ exists. We shall construct it inductively. Let $x_0=x$ and $\gamma_0=\operatorname{id}_x$. Assume that for some $i\in \N$, $x_i$ and $\gamma_i\in \Hom(x_i,x_{i-1})$ with $\mathbf t(\gamma_i)\geq k_0$  have been chosen. Use \ref{A:minimum} %
to get a minimizing pair of a point $x_{i+1}$ and a curve $\gamma_{i+1}\in\Hom(x_{i+1},x_i)$ such that $\mathbf t(\gamma_{i+1})\geq k_0$ and, since $\gamma_{i+1}$ is the minimizer, from \eqref{eq:dominated} we get
\begin{equation}
    \label{eq:calibration}u(x_{i})=u(x_{i+1})+A(\gamma_{i+1})+c_0\mathbf t(\gamma_{i+1}).
\end{equation}
Denote $\mathbf t_j\coloneqq\mathbf t(\gamma_j\circ\gamma_{j-1}\circ\dots\circ\gamma_1)=\sum_{i=1}^{j}\mathbf t(\gamma_{i})\geq k_0j$.

Let the curve $y\colon(-\infty,0]\to\categ$ be the concatenation of the curves $\bar\gamma_j$, so that its restrictions to $[-\mathbf t_j,0]$, $j>0$, are given by
\begin{align*}y_t=y(t)&=\overline{\gamma_j\circ\gamma_{j-1}\circ\cdots\circ\gamma_1}(\mathbf t_j+t),&&t\in[-\mathbf t_j,0],\\
&=\bar\gamma_k(\mathbf t_k+t),&&t\in[-\mathbf t_k,-\mathbf t_{k-1}].
\end{align*}
To define the category $\Gamma_x$, we first let the objects be 
\[\Gamma_x=\{y_t:t\leq 0\}=\bigcup_{i=1}^\infty \bar\gamma_i([0,\mathbf t(\gamma_i)]).\] 
Let $\Hom(\Gamma_x)$ be the set that contains all elements $\eta\in\Hom(\categ)$ such that there are some $j>0$ and some interval $I\subset[0,\mathbf t_j]$ with $\eta=\gamma_j\circ\dots\circ\gamma_1|_I$, which exists by \ref{A:restriction}.
For $s\leq 0$, observe that $y_s\in\Gamma_x$ is the point such that the single element $\gamma\in\Hom(y_s,x)\cap\Hom(\Gamma_x)$ satisfies $s=-\mathbf t(\gamma)$. It follows in particular that $y_{-\mathbf t_j}=x_j$ for $j\in\N$.

Let us show that \eqref{eq:calibrated} holds in $\Gamma_x$.
First, note that for each $i\in\mathbb N$ we have, by $i$ consecutive applications of \eqref{eq:calibration},
\begin{multline*}
    u(x_i)=u(x_{i-1})-A(\gamma_i)-c_0\mathbf t(\gamma_i)\\
    =u(x_{i-2})-A(\gamma_i\circ\gamma_{i-1})-c_0\mathbf t(\gamma_i+\gamma_{i-1})\\
    =\dots=
    u(x)-A(\gamma_i\circ\gamma_{i-1}\circ\dots\circ\gamma_1)-c_0\mathbf t_i;
\end{multline*}
in other words,
\[u(x)-u(x_i)=A(\gamma_i\circ\dots\circ\gamma_1)+c_0\mathbf t_i.\]
Let $s\leq t\leq 0$, and take $j$ such that $\mathbf t_j>-s$. Denote $\gamma=\gamma_j\circ\dots\circ\gamma_1$, and for a closed interval $I$, let $\gamma|_I$ be as in \ref{A:restriction}; in particular, we have that $\gamma|_{[-\mathbf t_j,0]}=\gamma$. Observe that
\begin{align*}
    u(x)-u(y_t)
    &=u(x)-u(x_j)+u(x_j)-u(y_t)\\
    &=A(\gamma)+c_0\mathbf t_j+u(x_j)-u(y_t)\\
    &\geq A(\gamma)+c_0\mathbf t_j-A(\gamma|_{[-\mathbf t_j,t]})-c_0(t+\mathbf t_j)\\
    &=A(\gamma|_{[t,0]})-c_0t,
\end{align*}
and similarly replacing $t$ by $s$.
It follows that
\begin{align*}
    u(y_t)-u(y_s)
    &=u(x)-A(\gamma|_{[t,0]})+c_0t-u(x)+A(\gamma|_{[s,0]})-c_0s\\
    &=A(\gamma|_{[s,t]})+c_0(t-s),
\end{align*}
which is precisely \eqref{eq:calibrated}.

For the converse, assume that $u$ is a weak \kam{} solution. From \eqref{eq:dominated} it follows that $u\leq \varphi_tu+c_0t$ for all $t>0$. To prove the opposite inequality, let now $t>0$ and $x$ be an object in $\categ$. Let $y_{-t}$ be the object in $\Gamma_x$ and $\gamma_{-t,0}$ the morphism in $\Hom_{\Gamma_x}(y_{-t},x)$ verifying \eqref{eq:calibrated}, that is,
\[u(x)-u(y_{-t})=A(\gamma_{-t,0})-c_0(0-t).\]
Then we have
\begin{align*}
 \varphi_tu(x)+c_0t&=\inf_{y\in \categ}u(y)+h_t(y,x)+c_0t\\
 &\leq u(y_{-t})+h_t(y_{-t},x)+c_0t\\
 &\leq u(y_{-t})+A(\gamma_{-t,0})-c_0(0-t)=u(x).\qedhere
\end{align*}
\end{proof}

\subsection{Weak \kamt{} machinery for noncompact metric spaces}
\label{sec:machinery}
The purpose of this section is to present and prove Theorem \ref{thm:mainlemma},
a result
that powers a rich range of weak \textsc{kam} theories. We will use the definitions and notations from Section \ref{sec:abstractsetting}.

Recall that a topological space is $\sigma$-compact if it can be covered by countably-many compact sets.

\begin{thm}\label{thm:mainlemma}
 Let $\categ$ be a Lagrangian category with action $A$ and time function $\mathbf t$. Assume that the class of objects of $\categ$ is a set that has the structure of a $\sigma$-compact metric space with distance function $\dist$, and that the finite action potential $h$ satisfies the technical hypotheses that follow:
 \begin{enumerate}[label=K\arabic*.,ref=K\arabic*]
  \item\label{it:lipschitzity} %
  There is some constant $P>0$ such that 
  \[h_{\dist(x,y)}(x,y)\leq P\dist(x,y)\]
  for all objects $x,y\in \categ$.
  In particular, $\Homt{\dist(x,y)}(x,y)$ is not empty.
  \item \label{it:superlinearity} 
  For every $k_1>0$ there is $k_2>0$ such that, 
  for all $x,y\in\categ$ and all $t>0$,
  \[k_1\dist(x,y)-k_2t\leq h_t(x,y).\]
  \item \label{it:sufficientlymanyshortcurves}
   There is some $Q>0$ such that, for all $x\in \categ$ and all $t>0$, we have
   \[h_t(x,x)\leq Qt.\]
 \end{enumerate}
 Then there are a locally Lipschitz function $u\colon \categ\to\R$ and a number $c_0\in\R$ such that, for all $t>0$, 
 \[u=\varphi_tu+c_0t.\]
 
\end{thm}
\begin{rmk}\label{rmk:weakkamsolution}
 By Lemma \ref{lem:amenablessolvable}, on a finely \kam{}-amenable category the conclusion of Theorem \ref{thm:mainlemma} is equivalent to $u$ being a weak \kam{} solution, as defined in Section \ref{sec:weakkamsolutions}.
\end{rmk}

Theorem \ref{thm:mainlemma} is %
a categorical formulation of a very slight generalization of the statement proved in \cite{madernafathi}. 
To prove it, we need a definition and an auxiliary lemma. 

Let $u\in C^0(\categ)$. For $c\in\R$, we say that $u$ is \emph{$c$-dominated} if, for every $t>0$,%
\[u(y)-u(x)\leq h_t(x,y)+ct.\]
We denote $\mathcal H(c)$ the set of $c$-dominated functions.

\begin{lem}\label{lem:auxlemmaLO}
Under the assumptions of Theorem \ref{thm:mainlemma} and with its notations, we have:
\begin{enumerate}[label=\roman*.,ref=(\roman*)]
 \item \label{it:lipschitzityofHc} The functions $u\in \mathcal H(c)$ are %
 uniformly Lipschitz with constant $P+c$.
 \item \label{it:nonemptyHc} There is $c_0>0$ such that the set $\mathcal{H}(c)$ is nonempty for all $c\geq c_0$.
 \item \label{it:phiinvolutionHc} $\varphi_s(\mathcal H(c))\subseteq \mathcal H(c)$.
 \item \label{it:continuity} $(t,u)\mapsto \varphi_t(u)$ is a continuous map on $[0,+\infty)\times\mathcal H(c)$, where %
 $\mathcal H(c)$ is given the topology of uniform convergence on compact sets.
\end{enumerate}
\end{lem}
\begin{proof}[Proof of Lemma \ref{lem:auxlemmaLO}]

The first assertion follows immediately from assumption \ref{it:lipschitzity}: 
if $u\in \mathcal H(c)$ and $x,y\in\categ$, and taking $t=\dist(x,y)$, we have
\[u(x)-u(y)\leq h_t(y,x)+ct\leq (P+c)\dist(x,y).\]
The other inequality can be obtained exchanging $x$ and $y$.

To see that the second assertion follows from \ref{it:superlinearity}, notice that if $u\colon\categ\to\R$ is Lipschitzian with Lipschitz constant $\leq k_1$, then there is $k_2>0$ such that, for all $t>0$, $x,y\in\categ$,
\[u(y)-u(x)\leq k_1\dist(x,y)\leq h_t(x,y)+k_2t,\]
whence $u\in\mathcal H(c)$ for all $c\geq k_2$. Thus to get one possible value for $c_0$, use $k_1=1$ and let $c_0=k_2$.

For the third assertion, observe first that $u\in \mathcal H(c)$ if, and only if, for all $x,y\in\categ$ and $t>0$,
\[u(x)\leq u(y)+h_t(y,x)+ct\]
which is true
if, and only if, (taking the infimum on the right hand side)
\begin{equation}\label{eq:characterizationofdomination}
u(x)\leq \varphi_tu(x)+ct.
\end{equation}
Moreover, since by the definition of $\varphi_t$ we clearly have that if $u\leq v$ then $\varphi_tu\leq\varphi_tv$, we see that  applying $\varphi_t$ to both sides of \eqref{eq:characterizationofdomination} and using that $\varphi_t(u+k)=\varphi_tu+k$ for $k\in\R$, we get 
\[\varphi_tu(x)\leq\varphi_{t}\varphi_t u(x)+ct,\]
whence $\varphi_tu\in\mathcal H(c)$ again.

To prove the last item, we will proceed in a way analogous to \cite[Proposition 3.3(3)]{madernafathi}. We will do this in three steps.

\vspace{2mm}
\emph{Step 1.} 
We will show that there is $\bar C(c)>0$ such that, for all $u\in\mathcal H(c)$,  
\[\varphi_tu(x)=\inf\{u(y)+h_t(y,x):y\in\categ,\;\dist(x,y)\leq \magicK(c)t\}.\]

To see why this is true, note that by \ref{it:sufficientlymanyshortcurves}, taking $x=y$ in the definition of $\varphi_tu$, we have
\begin{equation}\label{eq:whatwesaw}
 \varphi_tu(x)\leq u(x)+h_t(x,x)\leq u(x)+Qt.
\end{equation}

Thus, we have
\[\varphi_tu(x)=\inf\{u(y)+h_t(y,x):y\in\categ,\;u(y)+h_t(y,x)\leq u(x)+Qt\}.\]
Since $u\in\mathcal H(c)$, by item \ref{it:lipschitzityofHc} the latter condition boils down to
\[h_t(y,x)\leq u(x)-u(y)+Qt\leq (P+c)\dist(x,y)+Qt.\]
On the other hand, by \ref{it:superlinearity} we have that, if we let $k_1>P+c$, then there is $k_2>0$ such that
\[k_1\dist(x,y)-k_2t\leq h_t(y,x).\]
Putting these two together we get
\[(k_1-P-c)\dist(x,y)\leq (k_2+Q)t.\]
Thus, since $k_1\geq P+c$, so we can divide by $k_1-P-c$. This shows the claim above with     $\magicK(c)=\frac{k_2+Q}{k_1-P-c}$.

\vspace{2mm}
\emph{Step 2.} 
We show that, for fixed $t\geq 0$, $\varphi_t\colon\mathcal H(c)\to\mathcal H(c)$ is continuous.

Let $u,v\in\mathcal H(c)$. Let $x,y\in \categ $ be two points with $\dist(x,y)\leq \magicK(c)t$. Then we have
\begin{align*}
 \varphi_tv(x)&\leq v(y)+h_t(y,x)\\
 &\leq u(y)+h_t(y,x)+|u(y)-v(y)|\\
 &\leq u(y)+h_t(y,x)+\sup \{|u(w)-v(w)|:\dist(x,w)\leq \magicK(c)t\}.
\end{align*}
Take the infimum over $y\in\categ$ to get
\[\varphi_tv(x)\leq \varphi_tu(x)+
 \sup \{|u(w)-v(w)|:\dist(x,w)\leq \magicK(c)t\}.
 \]
Exchanging $u$ and $v$, get
\[|\varphi_tv(x)-\varphi_tu(x)|\leq \sup \{|u(w)-v(w)|:\dist(x,w)\leq \magicK(c)t\}.\]
Take the supremum letting $x$ run over a compact set $A\subset \categ$ to get
\begin{align*}
 \sup_A|\varphi_tv-\varphi_tu|&\leq \sup\{|u(w)-v(w)|:\dist(w,A)\leq \magicK(c)t\}\\
 &\leq \sup\{|u(w)-v(w)|(1+2(P+c)\magicK(c)t):w\in A\}
\end{align*}
since $u-v$ must be Lipschitzian with Lipschitz constant at most equal to the sum of the Lipschitz constants of $u$ and $v$, which are both bounded by $P+c$ by item \ref{it:lipschitzityofHc}. 
We can rewrite this as
\[\sup_A|\varphi_tv-\varphi_tu|\leq %
(1+2(P+c)\magicK(c)t)
\sup_A|u-v|.\]
This shows the continuity of $\varphi_t$ in the compact-open topology of $\mathcal H(c)$.

\vspace{2mm}
\emph{Step 3.}
We show the Lispchitz continuity of $t\mapsto \varphi_tu$.

It suffices to show that there is some $C>0$ such that $\|\varphi_su-\varphi_tu\|_\infty\leq C|s-t|$. By the semigroup property and item \ref{it:phiinvolutionHc}, we just need $\|u-\varphi_tu\|\leq Ct$ for $t>0$. Now, since $u\in\mathcal H(c)$, we have $u(x)\leq u(y)+h_t(y,x)+ct$. Taking the infimum over all $y$, this gives $u\leq \varphi_tu+ct$. We saw above \eqref{eq:whatwesaw} that $\varphi_tu\leq u+Qt$. Thus letting $C=\max(Q,c)$, we have what we need.

To finish the proof of \ref{it:continuity}, combine Steps 2 and 3 to get the continuity of $(t,u)\mapsto\varphi_t(u)$ on the required set: with $A$ as above,
\begin{align*}
    \sup_A|\varphi_tu-\varphi_sv|&\leq \sup_A|\varphi_tu-\varphi_su|+\sup_A|\varphi_su-\varphi_sv|\\
    &\leq C|t-s|+ (1+2(P+c)\bar C(c)s)\sup_A|u-v|.\qedhere
\end{align*}
\end{proof}

\begin{proof}[Proof of Theorem \ref{thm:mainlemma}.] The proof is essentially the same as that in \cite[Section 4]{madernafathi}.

We denote by $\widehat C^0(\categ)$ the quotient of the vector space $C^0(\categ)$ by its subspace of constant functions. If $\widehat q\colon C^0(\categ)\to \widehat C^0(\categ)$ is the quotient map, then since $\varphi_s(u+k)=k+\varphi_s(u)$ for $k\in\R$, the semigroup  $\varphi_s$ (cf.  Lemma \ref{lem:semigroup}) induces a semigroup on $\widehat C^0(\categ)$ that we denote $\widehat \varphi_s$ and satisfies $\widehat\varphi_s\circ\widehat q=\widehat q\circ\varphi_s$ for all $s>0$ as well as $\widehat \varphi_{s+t}=\widehat\varphi_s\circ\widehat\varphi_t$.

The topology on $\widehat C^0(\categ)$ is the quotient of the compact open topology (i.e., the topology of uniform convergence on compact sets). With this topology, the space $\widehat C^0(\categ)$ becomes a locally convex topological vector space. 

Denote by $\widehat {\mathcal H}(c)$ the image $\widehat q(\mathcal H(c))$ in $\widehat C^0(\categ)$. The subset $\widehat {\mathcal H}(c)$ of $\widehat C^0(\categ)$ is convex and compact. The convexity of $\widehat {\mathcal H}(c)$ follows from that of $\mathcal H(c)$. To prove that $\widehat{\mathcal H}(c)$  is compact, we introduce $C^0_x(\categ)$ the set of continuous functions $\categ\to\R$ vanishing at some fixed $x\in \categ$. The map $\widehat q$ induces a homeomorphism from $C^0_x(\categ)$ onto $\widehat C^0(\categ)$. Since $\mathcal H(c)$ is stable under addition of constants, its image $\widehat{\mathcal H}(c)$ is also the image under $\widehat q$ of the intersection $\mathcal H_x(c)=\mathcal H(c)\cap C^0_x(\categ)$. The subset $\mathcal H_x(c)$ is closed in $C^0(\categ)$ for the compact-open topology. Moreover, it consists of functions that all vanish at $x$ and are locally uniformly Lipschitz (because of Lemma \ref{lem:auxlemmaLO}\ref{it:lipschitzityofHc}).

Let us show that $\mathcal H_x(c)$ is compact.
Pick compact sets $K_1\subset K_2\subset\cdots\subset \categ$ such that $x\in K_i$ for all $i$ and $\categ=\bigcup_iK_i$, and let $\{f_i\}_i\subset\mathcal H_x(c)$ be a sequence of functions. 
From the Lipschitz version of the Arzel\`a-Ascoli theorem, it follows that there is a subsequence $\{f_{i^1_j}\}_{j}$ convergent in $K_1$. We iteratively apply the same theorem to pass to subsequent subsequences $\{f_{i^m_j}\}_{j}\subset \{f_{i^{m-1}_{j}}\}_{j}$ that converge on $K_m$ for each $m\in\N$. Taking the diagonal sequence $\{f_{i^j_j}\}_j$ we ensure convergence throughout $\categ$.
This means that $\mathcal H_x(c)$ is sequentially compact. Since the space $C^0(\categ)$ is metrizable by 
\[\dist(f,g)=\sum_{i}\frac1{2^{i}}\frac{\sup_{K_i}\|f-g\|}{1+\sup_{K_i}\|f-g\|},\] 
$\mathcal H_x(c)$ is also compact.

The restriction of $\widehat q$ to $\mathcal H_x(c)$ induces a homeomorphism onto $\widehat{\mathcal H}(c)$.
As a first consequence we conclude that if 
\[c_0=\inf\{c\in \R|\mathcal H(c)\neq \emptyset\}\]
then $\bigcap_{c>c_0}\widehat{\mathcal H}(c)\neq \emptyset$ as the intersection of a decreasing family of compact nonempty subsets (cf. Lemma \ref{lem:auxlemmaLO}\ref{it:nonemptyHc}). It follows that $\mathcal H(c_0)$ is also nonempty because it contains the nonempty subset $\widehat q^{-1} \left[\bigcap_{c>c_0}\widehat{\mathcal H}(c)\right]$.

We have that $\widehat\varphi_s(\widehat q(u))=\widehat q[\varphi_su-\varphi_su(x)]$ for $u\in\mathcal H_x(c)$ because $\varphi_su(x)$ is a constant so $\widehat q$ maps it to 0.
Since, by Lemma \ref{lem:auxlemmaLO}\ref{it:continuity}, the map
\begin{align*}
[0,+\infty)\times\mathcal H_x(c)&\to\mathcal H_x(c)\\
(s,u)&\mapsto \varphi_s(u)-\varphi_s(u)(x)
\end{align*}
is continuous, we conclude that $\widehat\varphi_t$ induces a continuous semigroup of $\widehat{\mathcal H}(c)$ into itself (cf. Lemma \ref{lem:auxlemmaLO}\ref{it:phiinvolutionHc}). Since this last subset is a nonempty convex compact subset of a locally convex topological vector space $\widehat C^0(\categ)$, we can apply the Schauder-Tychonoff theorem \cite[pages 414--415]{dugundji} to conclude that, for each $t>0$, $\widehat \varphi_t$ has a fixed point $\widehat u_t$ in $\widehat{\mathcal H}(c)$ if $\mathcal H(c)\neq \emptyset$, that is, for every $c\geq c_0$. Observing that for every rational $\frac pq\in\Q$ we have 
\[\widehat u_{\frac1q}=\underbrace{\widehat \varphi_{\frac1q}\circ \widehat \varphi_{\frac1q}\circ \dots\circ \widehat \varphi_{\frac1q}}_{p}(\widehat u_{\frac1q})=\widehat\varphi_{\frac pq}(\widehat u_{\frac1q})\]
so that $\widehat u_{\frac1q}=\widehat u_{\frac pq}$.
Using a density argument, we conclude that there is in fact a fixed point $\widehat u\in\widehat H(c)$ common to all the maps $\widehat\varphi_t$, $t>0$.

If we let $u\in \mathcal H(c)$ be such that $\widehat q (u)=\widehat u$, then the fact that $\widehat u$ is a fixed point for $\widehat \varphi_s$ means that  $\varphi_tu=u+a(t)$. Using that $\varphi_t$ is a semigroup, we get that $a(t)=a(1)t$. The equality $u=\varphi_tu-a(1)t$ shows that for all $x,y\in \categ$
\[u(y)-u(x)\leq h_t(x,y)-a(1)t.\]
Hence, from our choice of $c_0$, $-a(1)\geq c_0$. Since $u\in \mathcal H(c_0)$, we must have $u\leq \varphi_tu+c_0t$, which gives $\varphi_tu-a(1)t\leq \varphi_tu+c_0t$ for all $t\geq 0$, and $-a(1)\leq c_0$. We conclude that $-a(1)=c_0$.
\end{proof}
 
\subsection{Examples}
\label{sec:examples}

\subsubsection{Weak \kamt{} for Lagrangian dynamics.}\label{sec:lagrangian}
Our first example concerns the setting originally treated by Fathi--Siconolfi \cite{fathibook}. We take a compact Riemannian manifold $(M,g)$ and we let the objects of the category $\categ_{\mathrm{FS}}$ be the points of the manifold $M$, whose distance function is the one induced by the Riemannian metric $g$.  The morphisms $\Hom(x,y)$ of $\categ$ are simply the sets of absolutely continuous curves $\gamma\colon[0,T]\to M$ defined on an interval $[0,T]$ of arbitrary length, and with $\gamma(0)=x$ and $\gamma(T)=y$. Their composition $\gamma\circ\eta$ is defined  to be the concatenation $\gamma*\eta$, that is, the curve traversing $\gamma$ and then $\eta$.
The tangent bundle $TM$ is endowed with a Lagrangian function $L\colon TM\to\R$ that is smooth and satisfies the Tonelli conditions of strict convexity (the Hessian of $L(x,v)$ with respect to the velocity variables $v$ is everywhere positive definite) and superlinearity on the fibers of $TM$, which is the requirement that for every $A>0$ there be $B>0$ such that, for all $(x,v)\in TM$,
\[L(x,v)\geq A\|v\|-B.\]
The action $A\colon\Hom(\categ_{\mathrm{FS}})\to\R$ is defined by 
\[A(\gamma)=\int_0^TL(\gamma(t),\gamma'(t))\,dt\]
for a curve $\gamma\colon[0,T]\to M$, and the time function $\mathbf{t}\colon\Hom(\categ_{\mathrm{FS}})\to[0,+\infty)$ is given by $\mathbf t(\gamma)=T$. 

Since category $\categ_{\mathrm{FS}}$ satisfies  \ref{A:curves} and \ref{A:division}, it is Lagrangian by Lemma \ref{lem:sufflagrangian}. It is also finely \kam{} amenable. Indeed,  \ref{A:lagrangian}--\ref{A:restriction} are clear. To show that \ref{A:minimum} is also true, obeserve that by the superlinearlity of $L$, the curves $\gamma$ with $\gamma(0)=x$ and $A(\gamma|_{[-t,0]})\leq h_t(\gamma(-t),x)+1$ are uniformly Lipschitz, so they are compact by Arzel\`a-Ascoli. It is shown in \cite{fathibook} that $A$ is lower semicontinuous. Whence the minimum in \ref{A:minimum} is attained.

Theorem \ref{thm:mainlemma} can be applied because $\categ_{\mathrm{FS}}$ is finely \kam{} amenable and satisfies the hypotheses of the theorem, as we shall now explain. To see that it verifies \ref{it:lipschitzity}, take a curve $\gamma\in \Hom(x,y)$ with $\|\gamma'(t)\|=1$, so that the arclength of $\gamma$ is $\dist(x,y)=\mathbf t(\gamma)$ and $A(\gamma)\leq P\dist(x,y)$ with
\[P=\sup_{\substack{(x,v)\in TM\\\|v\|\leq 1}}L(x,v).\]
Hypothesis \ref{it:superlinearity} follows immediately from the superlinearity of $L$. Finally, \ref{it:sufficientlymanyshortcurves} is shown to be true by taking, for $x\in M$, the constant curve $\gamma(s)=x$, $s\in[0,t]$, which satisfies $A(\gamma)=L(x,0)t$, so we can take $Q=\max_{x\in M}L(x,0)$.

The function $u$ rendered by Theorem \ref{thm:mainlemma} can then be interpreted as a function on $M$, and since
\begin{align*}
 u(y)&=\varphi_Tu(y)+c_0T\\
 &=\inf_{x\in \categ_{\mathrm{FS}}}u(x)+h_T(x,y)+c_0T\\
 &=\inf_{x\in \categ_{\mathrm{FS}}}u(x)+\inf_{{\gamma\in\Homt{T}(x,y)}}A(\gamma)+c_0T\\
 &=\inf_{x\in M}u(x)+\inf_{\substack{\gamma\colon[0,T]\to M\\\gamma(0)=x,\gamma(T)=y}}\int_0^TL(\gamma(t),\gamma'(t))\,dt+c_0T,
\end{align*}
we have 
\[u(y)-u(x)\leq\int_0^TL(\gamma(t),\gamma'(t))\,dt+c_0T\]
for all $\gamma\in \Homt{T}(x,y)$. %
Denoting by $du$ the differential of $u$, it follows that 
\begin{equation}\label{eq:HJ}
 du_xv\leq L(x,v)+c_0
\end{equation}
for  any point $x\in M$ where $u$ is differentiable, and $v\in T_xM$. If we let $H$ be the Hamiltonian function associated to $L$ by the Legendre--Fenchel transform,
\[H(x,p)=\sup_{v\in TM}p(v)-L(x,v),\]
then \eqref{eq:HJ} can be written 
\[H(du)\leq c_0,\]
which is to say that $u$ is a subsolution of the Hamilton--Jacobi equation.
As the category $\categ_{\mathrm{FS}}$ is finely \kam{}-amenable, as a consequence of \eqref{eq:calibrated} the Fathi--Siconolfi \cite{fathibook} weak \kam{} theory actually gives us a viscosity solution $u$ of the Hamilton--Jacobi equation, $H(du)=c_0$. Here, ``viscosity'' means that, while $u$ is not a classical solution of this partial differential equation, it does enjoy certain desirable regularity properties; see \cite{fathibook,crandallishiilions}.

Many properties of the weak \kam{} solution $u$ and the curves realizing \eqref{eq:calibrated} have been studied in this context; see \cite{fathibook}.
Among many other generalizations, the case in which $M$ is $\sigma$-compact instead of compact was treated in \cite{madernafathi}.

\subsubsection{Weak \kamt{} for optimal control problems.}\label{sec:optimalcontrol}
Other, more general, versions of weak \kam{} arise in the contexts of various optimal control problems; see for example \cite{agrachev2010continuity}. For instance, if we consider the case of a category $\categ_{\mathrm{OC}}$ whose class of objects is given by the points of a $\sigma$-compact manifold $M$, and we also have a set $\mathcal U$ of controls and a function $f\colon M\times\mathcal U\to TM$ associating a velocity $f(x,u)\in T_xM$ to each pair $(x,u)\in M\times \mathcal U$. We also assume that we have a continuous Lagrangian density $L\colon M\times\mathcal U\to\R$ satisfying some technical conditions.

The morphisms $\gamma\in \Hom(x,y)$ are curves $\gamma\colon[0,T]\to M\times\mathcal U$ such that, if $\pi_M\colon M\times \mathcal U\to M$ is the projection, $\pi_M\circ\gamma$ is absolutely continuous and $(\pi_M\circ\gamma)'(t)=f(\gamma(t))$ for almost every $t\in [0,T]$, and $L\circ\gamma$ is integrable. For such $\gamma$, we define $A(\gamma)=\int_0^T L(\gamma(t))\,dt$  and we set $\mathbf t(\gamma)=T$. 

Then, in a manner similar to our explanation of Section \ref{sec:lagrangian}, the function obtained by the weak \kam{} machinery can be identified with the value function and is a viscosity solution of the Hamilton--Jacobi--Bellman equation. See for example \cite{akiandavidgaubert} and the references therein.

Since the category $\categ_{\mathrm{OC}}$ satisfies \ref{A:curves} and \ref{A:division}, by Lemma \ref{lem:sufflagrangian}, it is Lagrangian.
It is finitely \kam{} amenable for example if both $M$ and $f(M\times \mathcal U)$ are compact, from where \ref{A:minimum} follows, 
 and \ref{A:lagrangian}--\ref{A:restriction} are clear. 

\subsubsection{Weak \kamt{} for mass transportation.} \label{sec:masstransport}
This context has been explored for example in \cite{bernardbuffoni}, whose conclusions we now try to paraphrase.
We remark that the following description is not very detailed as the resulting theory is essentially the $n=1$ case of the one developed in Section \ref{sec:skeleton}.

Let $M$ be a compact, connected manifold, and let the objects of $\categ_{\mathrm{MT}}$ be the set of compactly supported Radon probability measures on $M$. We define, for any two $\nu_1,\nu_2\in\categ_{\mathrm{MT}}$, $\Hom(\nu_1,\nu_2)$ to be the set of measures $\mu$ on $M\times M$ with marginals $\pi^1_*\mu=\nu_1$ and $\pi^2_*\mu=\nu_2$, where $\pi^1$ and $\pi^2$ are the projections of $M\times M$ onto the corresponding copies of $M$. In this context we have a cost function $c\colon M\times M \to \R$ that allows us to define the action $A(\mu)=\int_{M\times M} c(x,y)\,d\mu(x,y)$, and if we have a Riemannian metric on $M$ inducing the distance, we can define the time function $\mathbf t(\mu)=\int_{M\times M}\dist(x,y)\,d\mu(x,y)$. The function $u$ given by the weak \kam{} theorem satisfying \eqref{eq:fixedpoint} can be better interpreted in an equivalent context.

With some technical assumptions, this context can be shown to be equivalent to the following \cite{bernardbuffoni}: Let $\tilde{\mathcal C}_{\mathrm{MT}}$ be the category with the same objects as $\categ_{\mathrm{MT}}$, but now we define, for compactly supported probability measures $\nu_1$ and $\nu_2$ on $M$, $\Hom(\nu_1,\nu_2)$ to be the set of compactly supported measures $\mu$ on $TM$ such that, for all $f\in C^\infty(M)$, 
\[\int_{TM}df_x\,d\mu=\int_Mf\,d\nu_1-\int_Mf\,d\nu_2.\]
Assume that $A$ can be written as $A(\mu)=\int_{TM}L\,d\mu$ for some function $L\colon TM\to\R$ and $\mathbf t(\mu)=\int_{TM}\|v\|\,d\mu(v)$. This would be for example the case if $c(x,y)=\dist(x,y)$ on a Riemannian manifold $(M,g)$, in which case $L(x,v)= {g_x(v,v)}$ is the square of the norm induced by the Riemannian metric $g$ of the corresponding tangent vector $v$.  In general, technical assumptions of convexity and superlinear growth on the function $L$ are necessary for the theory to hold. The function $u$ given by the weak \kam{} theorem (which is a function on the space of probabilities on $M$) turns out to correspond to a function on $M$ that is a viscosity solution of the Hamilton--Jacobi equation for the Hamiltonian associated to $L$; this basically means that the optimal transport trajectories for the probabilities are precisely the geodesics of the Lagrangian system associated to $L$, and the evolution of the probabilities, as they are transported, can be understood as a flow along these geodesics.

\subsubsection{Other examples}\label{sec:slicesintro}

Another vein that has been developed extensively and we will not attempt to fully describe here is the discrete setting and on graphs; see for example \cite{zavidovique2012strict, siconolfi2017global}.

In the theory we want to develop in this paper, we will consider submanifolds of a manifold $M$, and $\categ_{\mathrm{HM}}$ and $M$ will not coincide. Roughly speaking, the role of $\categ_{\mathrm{HM}}$ will be played by the set of $(n-1)$-dimensonal slices without boundary of $n$-dimensional submanifolds of $M$. %
The details of this approach will be developed in Section \ref{sec:slices} and the significance of the function $u$ will be explored in Section \ref{sec:coarsechar}.
 
\section{Weak \kamt{} for normal currents and holonomic measures}%
\label{sec:slices}
Let $M$ be a connected $C^\infty$ manifold of dimension $d>0$ without boundary, and let $g$ be a Riemannian metric on $M$. Denote by $\|\cdot\|$ and by $\dist_M(\cdot,\cdot)$, respectively, the norm on $T_xM$ and distance on $M$ induced by $g$. 

For $0<n\leq d$ we let $T^nM$ be the Whitney sum of $n$ copies of the tangent bundle $TM$, so that the fiber
\[T^n_xM=\underbrace{T_xM\oplus T_xM\oplus\dots\oplus T_xM}_n\]
is a vector space of dimension $nd$. %
We will denote a point in $T^nM$ by $(x,v_1,\dots,v_n)$ where $x\in M$ and $v_i\in T_xM$.

For vectors $v_1,\dots,v_n\in T_xM$, denote by $v_1\wedge\dots\wedge v_n$ their antisymmetric product. Given local coordinates $x_1,\dots,x_n$ on an open set of $M$ containing a point $x$, the vectors $\partial/\partial x_1,\dots,\partial/\partial x_n$ form a basis of $T_xM$ and the covectors $dx_1,\dots,dx_n$, satisfying 
\[dx_i(\partial /\partial x_j)=\begin{cases} 0,& i\neq j,\\1,&i=j,\end{cases}\]
form a basis of $T^*_xM$. 

For $0\leq k\leq d$, let $\Omega^k(M)$ be the set of $C^\infty$ differential forms of order $k$ on $M$. In particular $\Omega^0(M)=C^\infty(M)$. A form $\omega\in \Omega^k(M)$ can be written, locally on an open set $U\subseteq M$ diffeomorphic to a ball, as 
\[\omega(x,v_1,\dots,v_k)=\sum_{I}g_I(x)dx_I(v_1,\dots,v_k),\]
where the sum is taken over all subsets $I\subset\{1,2,\dots,d\}$ of cardinality $k$, $g_I\in C^\infty(U)$ and, for $I=\{i_1,\dots,i_k\}$,
\[dx_I(v_1,\dots,v_k)=dx_{i_1}\wedge dx_{i_2}\wedge \dots\wedge dx_{i_k}(v_1,\dots,v_k)=\det(v_{j,i_\ell})_{j,\ell=1}^k,\]
The topology of $\Omega^k(M)$ is induced by the seminorms 
\[|\omega|_{U,m}=\sum_{|J|\leq m}\sup_{x\in U}|\partial^J(g_I(x))_I|\]
where $U$ is a precompact subset of $M$, $m>0$, and the sum is taken over all multiindices $J=(j_1,\dots,j_d)$ with $|J|=\sum_ij_i\leq m$. Observe that a differential form of order $k$ is a function on $T^kM$. 

 A differential form $\omega$ of order $k$ is locally Lipschitz if it can be written locally on compact charts $U\subseteq M$ as $\omega=\sum_{I}g_Idx_I$ with $g_I\colon U\to\R$ Lipschitz continuous, local coordinates $x_1,\dots,x_n$ on $U$ and $dx_I=dx_{i_1}\wedge\dots\wedge dx_{i_k}$, $I=\{i_1<\dots<i_k\}\subset\{1,\dots,d\}$.

A \emph{current} $T$ of dimension $k$ is a continuous functional $T\colon \Omega^k(M)\to\R$. The \emph{boundary} $\partial T$ of a $k$-dimensional current $T$ is a $(k-1)$-dimensional current defined by 
\[\partial T(\omega)=T(d\omega),\quad \omega\in\Omega^{k-1}(M).\]
Since $dd\omega=0$ for all $\omega\in \Omega^k(M)$, we also have $\partial\partial T=0$ for all currents $T$.
A compactly-supported Radon measure $\mu$ on $T^nM$ induces a current $T_\mu$ by
\begin{equation}\label{eq:Tmu}
 T_\mu(\omega)=\int_{T^nM}\omega\,d\mu=\int_{T^nM}\omega_x(v_1,\dots,v_n)\,d\mu(x,v_1,\dots,v_n),\quad\omega\in \Omega^n(M).
\end{equation}
A current $T$ that can be represented as $T=T_\mu$ for some compactly-supported Radon measure $\mu$ is said to be \emph{normal}.

\subsection{Weak KAM theory for normal currents without boundary}\label{sec:skeleton}
We will first discuss a stripped-down theory that essentially deals with bare cobordisms, and in the next subsection we will explain how this can be enriched with a set of objects, namely, the argute slices, that make it more evident how the slices fit together and allow for the construction of a finely \kam{} amenable category.

\begin{defn}\label{def:CHM}Let $\categ_{\mathrm{HM}}$ be the category whose objects are the normal currents $T$ of dimension $n-1$ with null boundary $\partial T=0$. This means that $T$ should be induced by a finite, nonnegative, compactly-supported Radon measure $\nu$ on $T^{n-1}M$, that is, $T=T_\nu$ as defined by \eqref{eq:Tmu}. The morphisms of $\categ_{\mathrm{HM}}$ are, for each pair of currents $T_1$ and $T_2$, the collections $\Hom(T_1,T_2)$ of compactly-supported, nonnegative, Radon measures $\mu$ on $T^nM$ that satisfy
\[\int_{T^nM}d\omega\,d\mu=T_2(\omega)-T_1(\omega)\quad\textrm{for all $\omega\in\Omega^{n-1}(M)$}.\]
In other words, $\partial T_\mu=T_2-T_1$. Composition of $\mu_1\in \Hom(T_1,T_2)$ and $\mu_2\in\Hom(T_2,T_3)$ is defined by 
\[\mu_2\circ\mu_1=\mu_1+\mu_2.\]
We will refer to the objects of $\categ_{\mathrm{HM}}$ as \emph{slices}, and to the morphisms as \emph{holonomic measures}.
\end{defn}

Observe that holonomic measures also induce normal currents, and that the topology of $M$ may cause $\Hom(x,y)$ to be empty for certain pairs $x,y\in\categ_{\mathrm{HM}}$; this is why we pass to a (not necessarily unique) maximal subcategory $\tilde\categ_{\mathrm{HM}}$ of $\categ_{\mathrm{HM}}$ on which all morphism classes $\Hom(x,y)$ are nonempty. The maximality of $\tilde\categ_{\mathrm{HM}}$ here means that if we added any object $z\in \categ_{\mathrm{HM}}$ to $\tilde\categ_{\mathrm{HM}}$, then for some $x\in \tilde\categ_{\mathrm{HM}}$ we would have $\Hom(x,z)=\emptyset$ or $\Hom(z,x)=\emptyset$.

Given a continuous function $L\colon T^nM\to\R$, we let, for $\mu\in\Hom(T_1,T_2)$,
\[A(\mu)=\int_{T^nM}L\,d\mu\]
and
\[\mathbf t(\mu)=\int_{T^nM}1\,d\mu=\mu(T^nM).\]
 With these definitions, $\tilde\categ_{\mathrm{HM}}$ satisfies \ref{A:curves}--\ref{A:division}, so by Lemma \ref{lem:sufflagrangian} it is a Lagrangian category.

We give the collection of objects of $\categ_{\mathrm{HM}}$ the structure of a metric topological space by letting
\[\dist(T_1,T_2)=\inf_{\mu\in\Hom(T_1,T_2)}\int_{T^nM}\vol\,d\mu,\quad T_1,T_2\in \categ_{\mathrm{HM}}, \]
where 
\[\vol(x,v_1,\dots,v_n)=\sqrt{\det\left(g_x(v_i,v_j)\right)_{i,j=1}^n}.\] 
In geometric measure theory, $\dist$ is known as the \emph{flat distance}. 

An important property of the morphisms of $\categ_{\mathrm{HM}}$ is that they are \emph{reversible}. Indeed, to a morphism $\mu\in\Hom(S,T)$ corresponds to at least one morphism in $\Hom(T,S)$ that can be obtained by choosing a Borel-measurable mapping $\Psi\colon T^nM\to\T^nM$ such that, for each $(x,v_1,\dots,v_n)$ there is some $1\leq j\leq n$, $j=j(x,v_1,\dots,v_n)$, such that 
\[\Psi(x,v_1,\dots,v_n)=(x,v_1,\dots,v_{j-1},-v_j,v_{j+1},\dots,v_n),\] 
so that the pushforward measure $\Psi_*\mu$ satisfies 
\[\int_{T^nM}d\omega\,d\Psi_*\mu=-\int_{T^nM}d\omega\,d\mu=S(\omega)-T(\omega)\] 
and $\Psi_*\mu\in\Hom(T,S)$. In particular, this means that $\Hom(S,T)=\Hom(T,S)$.
Also, 
\[\int_{T^nM}\vol\,d\mu=\int_{T^nM}\vol\,d\Psi_*\mu.\]

With the metric $\dist$, the subset $\tilde\categ_{\mathrm{HM}}\subseteq\categ_{\mathrm{HM}}$ is closed; indeed, given $T\in \tilde\categ_{\mathrm{HM}}$ and given $S$ in the closure of $\tilde\categ_{\mathrm{HM}}$, we have $\dist(S,T)<+\infty$, which means that either $\Hom(S,T)$ or $\Hom(T,S)$ are not empty; since each morphism $\mu\in\Hom$ is reversible, in fact both $\Hom(S,T)$ and $\Hom(T,S)$ are nonempty, and $S\in \tilde\categ_{\mathrm{HM}}$.

\begin{lem}\label{lem:sigmacompact}
 With the topology induced by the metric $\dist$, $\categ_{\mathrm{HM}}$ and  $\tilde\categ_{\mathrm{HM}}$ are $\sigma$-compact.
\end{lem}

\begin{proof}
 Let $K_1\subseteq K_2\subseteq\dots\subseteq M$ be a family of nested compact balls of increasing radius, and $M=\bigcup_iK_i$.
 It follows from the Compactness Theorem for Normal Currents (see for example \cite[Theorem 1.4]{ambrosio2013compactness}, or \cite[Theorem 5.2]{ambrosio2000currents} for a very general version) 
 that the closed set $C_\ell$ of normal currents $T_\nu\in\categ_{\mathrm{HM}}$ associated to measures $\nu$ and supported in $K_\ell$ and with mass $\mathbf M(T_\nu)=\int\vol\,d\nu$ bounded by $\ell$, is sequentially compact in the weak* topology. By \cite[Corollary 7.3]{federerfleming}, the sets $C_\ell$ are also compact in the topology induced by the flat distance, so $\categ_{\mathrm{HM}}=\bigcup_\ell C_\ell$ is $\sigma$-compact. Since $\tilde\categ_{\mathrm{HM}}$ is a closed subset, it is $\sigma$-compact as well.
\end{proof}

In order to ensure we can find a weak \kam{} theory in this context through the application of Theorem \ref{thm:mainlemma}, we need some technical assumptions on $L$ to ensure that $h_t$ will satisfy the hypotheses of the theorem. 
Our choice of assumptions on  $L\colon T^nM\to\R$ is the following:
\begin{enumerate}[series=Lassmp,label=E\arabic*.,ref=E\arabic*]
 \item\label{Lfirst}
\label{L1} There is some $P>0$ such that
\[L(x,0,\dots,0)\leq P,\qquad L(x,v_1,\dots,v_n)\leq P,\]
for all $(x,v_1,\dots,v_n)\in T^nM$ with unit volume $\vol(x,v_1,\dots,v_n)=1$ and with $\|v_j\|= 1$, $j=1,\dots,n$, in other words, for all orthonormal frames $v_1,\dots,v_n$.
\item \label{L2} \label{Llast} 
Given $x_0\in M$ and $k_1>0$, there is $k_2>0$ verifying,  for all $(x,v_1,\dots,v_n)\in T^nM$, 
\[k_1(\dist_M(x_0,x)+\|v_1\|+\dots+\|v_n\|)-k_2\leq L(x,v_1,\dots,v_n).\]
Observe that, if this is true for one $x_0$, then it will be true for all $x\in M$ (perhaps with different constants $k_2$).
\end{enumerate}
Examples of Lagrangians satisfying \ref{Lfirst}--\ref{Llast} include the family \[L(x,v_1,\dots,v_n)=a(x)+\max(f(x)\vol(x,v_1,\dots,v_n),h(x)(\|v_1\|+\dots+\|v_n\|)^{b(x)})\]
with $a,b,f,h\in C^0(M)$,  $h(x)>0$ and $b(x)> 1$ everywhere, and 
\[\lim_{\dist_M(x,x_0)\to+\infty}\frac{a(x)}{\dist_M(x,x_0)}\to+\infty.\]

With these assumptions, we have

\begin{prop}[Weak \kam{} theorem]\label{prop:lamekam}
 Assume $L$ satisfies \ref{Lfirst}--\ref{Llast}. Then there is a Lipschitz function $u\colon\categ_{\mathrm{HM}}\to\R$ satisfying \eqref{eq:fixedpoint}.
\end{prop}
\begin{proof}
By Lemma \ref{lem:sigmacompact}, $\categ_{\mathrm{HM}}$ is $\sigma$-compact.
Lemma \ref{lem:lipschitzity} shows that assumptions \ref{it:lipschitzity}--\ref{it:sufficientlymanyshortcurves} of Theorem \ref{thm:mainlemma} are verified, which gives the proposition.%
\end{proof}

\begin{lem}\label{lem:lipschitzity} 
We have that, for $L$ as above, 
\begin{align*}
\text{\ref{L1}}  & \implies \text{\ref{it:lipschitzity} and \ref{it:sufficientlymanyshortcurves}},\\
\text{\ref{L2}} & \implies \text{\ref{it:superlinearity}}.%
\end{align*}
\end{lem}
\begin{proof}
 To prove the first assertion, assume that $L$ satisfies \ref{L1}.
  Let $P>0$ be as in \ref{L1}.
 Apply Lemma \ref{lem:distanceachieved} with $t=\dist(T_1,T_2)$ to obtain a measure  $\mu\in \Homt{\dist(T_1,T_2)}(T_1,T_2)$ supported on points $(x,v_1,\dots,v_n)\in T^nM$ with $\vol(v_1,\dots,v_n)=1$ and $\|v_j\|=1$ for all $j=1,\dots,n$.
 Then
 \[A(\mu)=\int_{T^nM}L\,d\mu\leq \int_{T^nM}P\,d\mu= P \dist(T_1,T_2).
 \]
 We thus get
 \[h_{\dist(T_1,T_2)}(T_1,T_2)\leq A(\mu)\leq P\dist(T_1,T_2),\]
 which is \ref{it:lipschitzity}.

  If \ref{L1} holds, then there is some $P>0$ such that, for all $\mu\in \Hom^t(T,T)$, $T\in \categ_{\mathrm{HM}}$, supported on the set of points of the form $(x,0,\dots,0)$, $x\in M$, we have
  \[\int_{T^nM}L\,d\mu\leq \int_{T^nM}P\,d\mu=P\,\mathbf t(\mu)=Pt.\]
  Observe that every measure $\mu$ with $\mu(T^nM)=t$ supported on that set of points belongs to $\Homt{t}(T,T)$.
  Thus $\Homt{t}(T,T)$ is not empty and we obtain \ref{it:sufficientlymanyshortcurves}.
 
  Now assume that \ref{L2} is verified and let us show that \ref{it:superlinearity} follows. Pick $k_1>0$ and let $k_2>0$ be as in \ref{L2}, and we will show that \ref{it:superlinearity} is verified with the same $k_1$ and $k_2$. By Lemma \ref{lem:distanceachieved}, given $T_1,T_2\in\categ_{\mathrm{HM}}$, $t> 0$, there is some $\mu\in \Homt{t}(T_1,T_2)$, so
  \begin{align*}
   k_1\dist(T_1,T_2)-k_2t&=\int_{T^nM}k_1\vol-k_2\,d\mu\\
   &\leq \int_{T^nM}k_1(\|v_1\|+\dots+\|v_n\|)-k_2\,d\mu\\
   &\leq \int_{T^nM}L\,d\mu=A(\mu),
  \end{align*}
  so, taking the infimum over all $\mu$, we get $k_1\dist(T_1,T_2)-k_2t\leq h_t(T_1,T_2)$, which is \ref{it:superlinearity}.
\end{proof}
\begin{lem}\label{lem:distanceachieved}
 For each $t> 0$ and each pair $T_1,T_2\in \tilde \categ_{\mathrm{HM}}$  such that $\Hom(T_1,T_2)$ is nonempty, there is some $\mu\in \Homt{t}(T_1,T_2)$ such that $\int_{T^nM}\vol\,d\mu=\dist(T_1,T_2)$ and $\mu$ is supported on the set
 \begin{multline*}
 H=\{(x,v_1,\dots,v_n):\text{either } v_1=\dots=v_n=0,\\ \text{or } \vol(v_1,\dots,v_n)=\tfrac{\dist(T_1,T_2)}t 
 \text{ and } \|v_j\|= \left(\tfrac{\dist(T_1,T_2)}t\right)^{1/n},\;j=1,2,\dots,n\}
 \end{multline*}
\end{lem}
\begin{proof}
If $\dist(T_1,T_2)=0$, any measure $\mu$ supported on the zero section and of mass $t$ will be an element of $\Homt{t}(T_1,T_2)$ that will satisfy the requirements of the statement of the lemma, so let us assume that $\dist(T_1,T_2)>0$.

Let $\eta\in \Hom(T_1,T_2)$ be such that $\int_{T^nM}\vol\,d\eta= \dist(T_1,T_2)>0$, which exists by an application of the Compactness Theorem for Normal Currents \cite[Theorem 1.4]{ambrosio2013compactness}, \cite[Theorem 5.2]{ambrosio2000currents}.

Let $\Phi\colon T^nM\to H\subset T^nM$ be a measurable\footnote{Observe that $\Phi$ may not enjoy any linearity, symmetry, or continuity properties.} mapping such that $\Phi(T_xM)\subseteq T_xM$ for all $x\in M$ and, if $(x,u_1,\dots,u_n)=\Phi(x,v_1,\dots,v_n)$, then: 
\begin{itemize}
    \item if the linear span of $v_1,\dots,v_n$ is not $n$-dimensional, set $\Phi(x,v_1,\dots,v_n)=0$;
    \item otherwise, pick $u_1,\dots, u_n$ such that the linear span of $v_1,\dots,v_n$ coincides with the linear span of $u_1,\dots,u_n$, both induce the same orientation and
\begin{equation}\label{eq:orthonormal}
g_x(u_i,u_j)=\begin{cases}
    c^{2/n},&i=j,\\
    0,&i\neq j,
\end{cases},\qquad c=\frac{\dist(T_1,T_2)}{t}.
\end{equation}
\end{itemize}
In other words, to every set of $n$ vectors $v_1,\dots,v_n$ spanning an $n$-dimensional subspace of $T_xM$, $\Phi$ associates an orthogonal frame $u_1,\dots,u_n$ for that same subspace, where each $u_i$ is of length $\|u_i\|=c^{1/n}$ and $\vol(x,u_1,\dots,u_n)=c$. The vectors $u_1,\dots,u_n$ can be found by the usual orthonormalization procedure and mulitplying by $c^{1/n}$. Note that $\omega\circ\Phi=\frac{c}{\vol}\omega$ for $\omega\in\Omega^n(M)$, and $\vol\circ\Phi=c$.

Let $\mu$ be the measure on $T^nM$ defined by $\mu=\frac1c\Phi_*({\vol}\cdot\eta)$; in other words, for all $f\in C^0(T^nM)$, and using the shorthands $u$ and $v$ for $(u_1,\dots,u_n)$ and $(v_1,\dots,v_n)$, respectively,
\[\int_{T^nM} f(x,u)d\mu(x,u)=\frac1c\int_{T^nM}f\circ\Phi(x,v)\vol(x,v)d\eta(x,v).\]

With this definition it follows that, for an $n$-form $\omega\in\Omega^n(M)$,
\begin{multline*}
    \int_{T^nM}\omega(x,u)d\mu(x,u)
    =\frac1c\int_{T^nM}\omega\circ\Phi(x,v){\vol(x,v)}d\eta(x,v)\\
    =\frac1c\int_{T^nM}\frac{c}{\vol(x,v)}\omega(x,v)\vol(x,v)d\eta(x,v)
    =\int_{T^nM}\omega(x,v)d\eta(x,v).
\end{multline*}
This means in particular that $\mu\in\Hom(T_1,T_2)$.
Additionally, we have
\[\mathbf t(\mu)=\mu(T^nM)=\frac1c\int_{T^nM}\vol(x,v)d\eta(x,v)=\frac{\dist(T_1,T_2)}c=t,\]
so in fact $\mu\in\Homt{t}(T_1,T_2)$.
Finally, we have
\begin{multline*}
    \int_{T^nM}\vol(x,u)d\mu(x,u)=\frac1c\int_{T^nM}\vol\circ\Phi(x,v)\vol(x,v)d\eta(x,v)\\
    =\frac1c\int_{T^nM}c\vol(x,v)d\eta(x,v)
    =\int_{T^nM}\vol(x,v)d\eta(x,v)=\dist(T_1,T_2).
\end{multline*}
 Thus the required properties of $\mu$ have been established. 

\end{proof}

\subsection{Weak KAM for argute slices}

In this subsection we give definitions conceived with the intention of clarifying what a slice of a holonomic measure is, and how the slices fit together. To fix ideas, we first offer an example.

\begin{example}\label{ex:torus}
 Let $\T^d=\R^d/\Z^d$ be the $d$-dimensional torus, and consider $\T^2\subset \T^3$ parameterized by $\phi(x_1,x_2)=(x_1,x_2,0)\operatorname{mod}\Z^3$, $x_1,x_2\in\R$. The torus $\T^2$ can be encoded as the holonomic measure $\mu_\phi$ on $T^2\T^3$ given by 
\[\mu_\phi=(\phi,\tfrac{\partial\phi}{\partial x_1},\tfrac{\partial\phi}{\partial x_2})_*\lebesgue_{[0,1]\times[0,1]},\]
where $\lebesgue_{X}$ is the Lebesgue measure on $X$; in other words, for a measurable function $f\colon T^2\mathbb T^3\to\R$, we have
\[\int_{T^2\mathbf T^3} f\,d\mu=\int_0^1\int_0^1f(
\left(\begin{smallmatrix}
    x_1\\x_2\\0
\end{smallmatrix}\right),\left(\begin{smallmatrix}
    1\\0\\0
\end{smallmatrix}\right),\left(\begin{smallmatrix}
    0\\1\\0
\end{smallmatrix}\right))\,dx_1\,dx_2.\]
We want to think of $\T^2$ as being composed of the diagonal copies of $\T^1$ parameterized by $\gamma_t(s)=(s,t-s,0)\operatorname{mod}\Z^3$, $s,t\in[0,1)$; we understand each $\gamma_t$ as a slice of $\T^2$, and they fit together as the level sets of the mapping $(x_1,x_2,0)\mapsto x_1-x_2\mod \mathbb Z$. For our purposes, we cannot assume that we know in advance how these slices fit together, because we want to construct a theory in which a family of slices makes up a submanifold, in the sense that we recover the associated holonomic measure $\mu_\phi$. Thus, although we can encode the circles $\gamma_t$ with the measures $\mu_{\gamma_t}=(\gamma_t,\gamma'_t)_*\lebesgue_{[0,1]}$, this is not a satisfactory answer because the vector $\gamma'_t(s)=(1,-1,0)$ does not encode the information contained in the vectors $\tfrac{\partial\phi}{\partial x_1}=(1,0,0)$ and $\tfrac{\partial\phi}{\partial x_2}=(0,1,0)$. 

A solution to this problem is to consider a richer object that encodes not only the information of the partial derivatives of $\phi$ but also the differential $dt=dx_1+dx_2$ of the parameter $t$ that describes how they fit together. Hence, instead of $\gamma_t$ we take the measure $\nu_t$ on $T^2M\oplus T^*M$ defined by
\[\nu_t=(\gamma_t,\left(\begin{smallmatrix}
                   1\\0\\0
                  \end{smallmatrix}\right)
,\left(\begin{smallmatrix}
  0\\1\\0
 \end{smallmatrix}\right),dx_1+dx_2)_*\lebesgue_{[0,1]}.\]
 Then we have, forgetting the last piece of information by pushing forward with the projection $\pi^1\colon T^2\T^3\oplus T^*\T^3\to T^2\T^3$,
 \[\mu_\phi=\int_0^1\pi^1_*\nu_t\,dt,\]
 so the slices fit together correctly.
 At the same time, we may recover the 1-dimensional character of each slice; they act on differential forms of order 1, according to the following recipe: for each $\nu_t$, we set $\hat T_{\nu_t}$ to be the current given, for $\omega=g_1dx_1+g_2dx_2+g_3dx_3\in\Omega^1(\T^3)$, by
 \begin{align*}\hat T_{\nu_t}(\omega)&\coloneqq\int_{} \omega\wedge \mathbbm t(x,v_1,v_2)\,d\nu_t(x,v_1,v_2,\mathbbm t)\\
  &=\int_0^1\omega\wedge (dx_1+dx_2) (\gamma_t(s),\left(\begin{smallmatrix}
                   1\\0\\0
                  \end{smallmatrix}\right)
,\left(\begin{smallmatrix}
  0\\1\\0
 \end{smallmatrix}\right))\,ds\\
 &=\int_0^1 (g_1-g_2)dx_1\wedge dx_2(\gamma_t(s),\left(\begin{smallmatrix}1\\0\\0\end{smallmatrix}\right),\left(\begin{smallmatrix}0\\1\\0\end{smallmatrix}\right))\,ds\\
 &=\int_0^1(g_1-g_2)\circ\gamma_t(s)\,\det\left(\begin{smallmatrix}
     1&0\\0&1
 \end{smallmatrix}\right)ds\\
 &=\int_0^1(g_1-g_2)\circ\gamma_t(s)\,ds\\
 &=\int_0^1\omega_{\gamma_t(s)}(\gamma'_t(s))\,ds\\
 &=\int_{\gamma_t}\omega.
 \end{align*}
 So $\hat T_{\nu_t}(\omega)$ indeed encodes $\gamma_t$. In conclusion, the measures $\nu_t$ contain enough information to simultaneously recover both an infinitesimal 2-dimensional slice of $\mu_\phi$ and the currents induced by  $\gamma_t$, and in this way they have 1- and 2-dimensional character at the same time, and they carry information about how they can fit within the full holonomic measure $\mu_\phi$.
\end{example}

We will now give the full definition.
Let $\dot T^*M$ be the bundle on $M$ whose fiber at each $x\in M$ is the Alexandroff 1-point compactification $\dot T_x^*M=T_x^*M\cup\{\infty\}$ of the cotangent space $T_x^*M$.
Let $\mathcal X=T^nM\oplus \dot T^*M$ be the bundle whose fiber at $x$ is the product $T^n_xM\times( T_x^*M\cup\{\infty\})$, and denote by $\pi^1$ and $\pi^2$ the canonical projections
\[\xymatrix{
& \ar[ld]_{\pi^1} \mathcal X \ar[rd]^{\pi^2}\\
T^nM&&{\dot T^*M}.
}\]
We will denote a point in $\mathcal X$ by $(x,v_1,v_2,\dots,v_n,\mathbbm t)$ for $x\in M$, $v_i\in T_xM$, $\mathbbm t\in \dot T^*_xM$, so that
\begin{align*}
     &\pi^1(x,v_1,\dots,v_n,\mathbbm t)=(x,v_1,\dots,v_n),\\
     &\pi^2(x,v_1,\dots,v_n,\mathbbm t)=(x,\mathbbm t).
\end{align*}

\begin{defn}\label{def:slice}
 We define  \emph{argute slices} to be those compactly-supported, finite, Radon measures $\nu$ on $\mathcal X$ such that:
 \begin{itemize}
     \item 
        the linear functional
         \begin{align*}
             \Omega^n(M)&\to \R \\
             \eta&\mapsto \int_{\mathcal X}\eta(x,v_1,v_2,\dots,v_n)\,d\nu(x,v_1,\dots,v_n,\mathbbm t)
         \end{align*}
         is continuous, so that it defines a current of dimension $n$, 
     
    \item 
        the linear functional $\hat T_\nu$ defined by
        \begin{align*}
            \hat T_\nu\colon \Omega^{n-1}(M)&\to\R\\
            \omega&\mapsto\langle \hat T_\nu,\omega\rangle=\int_{\mathcal X} (\omega\wedge \mathbbm t)(x,v_1,v_2,\dots,v_n)\,d\nu(x,v_1,\dots,v_n,\mathbbm t)
        \end{align*}
        is continuous, so that it defines a current of dimension $n-1$,
    \item 
        the $(n-1)$-dimensional current $\hat T_\nu$ has null boundary, in other words, 
        \[\partial \hat T_\nu(d\theta)=\hat T_\nu(d\theta)=0,\qquad\theta\in \Omega^{n-2}(M).\]
 \end{itemize}
 Denote the set of argute slices by $\mathscr S$.

\end{defn}

As in Example \ref{ex:torus}, the $\mathbbm t$ factor is there to perform the transition between the $(n-1)$-dimensional slice and the $n$-dimensional current it is a slice of. 

\begin{defn}\label{def:curve}
 A \emph{curve of argute slices} is a family $(\nu_t)_{t\in I}$, such that
 \begin{enumerate}[label=C\arabic*.,ref=C\arabic*]
  \item $I\subseteq\R$ is a nonempty interval,
  \item\label{C:issliced} $\nu_t\in \mathscr S$ for all $t\in I$,
  \item each $\nu_t$ is a probability measure, that is, $\nu_t(\mathcal X)=1$,
  \item for $\omega\in \Omega^n(M)$, the function
    \begin{align*}
        I&\to\R\\
        t&\mapsto \int_{\mathcal X} \omega_x(v_1,\dots,v_n)\,d\nu_t(x,v_1,\dots,v_n,\mathbbm t)
    \end{align*}
    is integrable,
  \item for each $r,t\in I$, the linear functional on differential forms
    \begin{align*}
        \Omega^n(M)&\to \R\\
        \omega&\mapsto \int_r^t\int_{\mathcal X}\omega_x(v_1,\dots,v_n)\,d\nu_s(x,v_1,\dots,v_n,\mathbbm t)\,ds
    \end{align*}
    is continuous and defines an $n$-dimensional current,
  \item\label{C:boundary} for $r,t\in I$, the  measure $\mu_{r,t}=\int_{r}^t\pi^1_*\nu_s\,ds$ on $T^nM$
  is such that its associated current $T_{\mu_{r,t}}$ has boundary
  \[\partial T_{\mu_{r,t}}=\hat T_{\nu_t}-\hat T_{\nu_r}.\]
  In other words, $\mu_{r,t}\in\Hom(\hat T_{\nu_r},\hat T_{\nu_t})$ in $\categ_{\mathrm{HM}}$ (Definition \ref{def:CHM}). Explicitly, this means that, for $\eta\in \Omega^{n-1}(M)$,
  \[
    \int_{\mathcal X}d\eta\,d\mu_{r,t}=\,d\mu\int_r^t\int_{\mathcal X} d\eta\,d\nu_sds
    =\int_{\mathcal X}\eta\wedge\mathbbm t\,d\nu_t-\int_{\mathcal X}\eta\wedge\mathbbm t\,d\nu_r.
  \]
 \end{enumerate}

  For a curve of slices $\gamma=(\nu_t)_{t\in I}$, its \emph{associated holonomic measure $\mu_\gamma$}, is  defined by 
 \[\mu_\gamma=\int_I\pi^1_{*}\nu_t\,dt.\]
 It coincides with $\mu_{0,T}$ in the notation of \ref{C:boundary}.
\end{defn}

 With these definitions and given the same information we used for $\categ_{\mathrm{HM}}$, namely, a  connected Riemannian manifold $(M,g)$ and a Lagrangian function $L\colon T^nM\to\R$ satisfying \ref{Lfirst}--\ref{Llast}, we may define a Lagrangian category $\categ_{\mathrm{AS}}$ that is closely related to $\categ_{\mathrm{HM}}$. 

 \begin{defn}
 We let the class of objects of $\hat \categ_{HM}$ be the set $\mathscr S_N$ of argute slices, and we let $\Hom(\eta_1,\eta_2)$ be the set of curves $\gamma=(\nu_t)_{t\in I}$, $I=[\alpha,\beta]\subseteq\R$, of argute slices of mass one, $\nu(\mathcal X)=1$, joining $\nu_\alpha=\eta_1$ and $\nu_\beta=\eta_2$, and such that $L$ is integrable with respect to the associated holonomic measure $\mu_\gamma$, that is, $\int_{\mathcal X}|L|\,d\mu_\gamma<+\infty$. The action of $\gamma=(\nu_t)_{t\in I}\in\Hom(\nu_1,\nu_2)$ is given by
 \[A(\gamma)=\int_I\left[\int_{T^nM}L\,d(\pi^1_*\nu_t)\right]dt=\int_{T^nM}L\,d\mu_\gamma,\]
 and the internal time duration %
 \[\mathbf t(\gamma)=\mu_\gamma(T^nM)=|I|\]%
 coincides with the length of the interval $I$ %
 because each $\nu_t$ is a probability.

 We pick a subcategory $\tilde\categ_{\mathrm{AS}}\subseteq\categ_{\mathrm{AS}}$ that is maximal with respect to the property that $\Hom^t(x,y)=\Hom(x,y)\cap\mathbf t^{-1}(t)$ is nonempty for all $x,y\in \tilde\categ_{\mathrm{AS}}$ and all $t\geq 0$; this is analogous to taking $\tilde\categ_{\mathrm{HM}}\subseteq \categ_{\mathrm{HM}}$ as in the previous section. Maximality here means that any subcategory $\categ$, with $\tilde\categ_{\mathrm{AS}}\subseteq\categ\subseteq\categ_{\mathrm{AS}}$ that also satisfies that all $\Hom^t(x,y)$ are nonempty, must be $\categ=\tilde\categ_{\mathrm{AS}}$.
 \end{defn}

  Analogously to the metric structure in $\categ_{\mathrm{HM}}$, we introduce the following distance in $\categ_{\mathrm{AS}}$:
  \[\dist(\nu_1,\nu_2)=\inf_{\gamma\in \Hom(\nu_1,\nu_2)}\int_{T^nM}\vol\,d\mu_\gamma.\]
  By our choice of $\tilde\categ_{\mathrm{AS}}$, this distance is finite throughout that subcategory.
 
 To obtain a weak \kam{} theory, we will assume that $L$ satisfies \ref{Lfirst}--\ref{Llast}.
 The main result of this section is
 
 \begin{thm}\label{thm:weakkamsol}
  If the manifold $M$ is  connected and \ref{Lfirst}--\ref{Llast} hold,  then there is a weak \kam{} solution $u\colon\tilde\categ_{\mathrm{AS}} \to\R$, as defined in Section \ref{sec:weakkamsolutions}.
 \end{thm}
 
 \begin{rmk}
  Remark \ref{rmk:hamiltonjacobi} and Corollary \ref{cor:exactdescent} explain the way in which the function $u$ rendered by the theorem can be connected to an exact differential form $d\omega$ on $M$ that satisfies a sort of Hamilton--Jacobi equation, thus furthering the analogy with the original weak \kam{} theory \cite{fathibook}.
 \end{rmk}
 
\begin{proof}[Proof of Theorem \ref{thm:weakkamsol}]
 The manifold $M$ is $\sigma$-compact because, since it is connected,  we can fix a point $x_0\in M$ and express $M$ as a union of countably-many compact balls $M=\bigcup_{R\in \mathbb N}\overline B_{R}(x_0)$.
  By Lemma \ref{lem:kamamenable}, the category $\categ_{\mathrm{AS}}$ is Lagrangian, and it follows using also Lemma \ref{lem:weakkamcoarse} that we may apply Theorem \ref{thm:mainlemma} to  $\categ_{\mathrm{AS}}$, which gives %
 a function $u$ satisfying \eqref{eq:fixedpoint}. 
 Lemma \ref{lem:kamamenable} shows that $\tilde\categ_{\mathrm{AS}} $ is finely \kam{} amenable. By Lemma \ref{lem:amenablessolvable}, it follows that $u$ is a weak \kam{} solution.
\end{proof}

 \begin{lem}\label{lem:weakkamcoarse}
  Assume that  \ref{Lfirst}--\ref{Llast} hold.
  Then the objects of the category $\tilde\categ_{\mathrm{AS}}$ form a $\sigma$-compact metric space that satisfies hypotheses \ref{it:lipschitzity}--\ref{it:sufficientlymanyshortcurves}.%
 \end{lem}
 
 \begin{proof}[Proof of Lemma \ref{lem:weakkamcoarse}]
  Fix a point $x_0\in M$ and consider the closed balls $\overline B_R(x_0)\subseteq M$ for $R\in \mathbb N$; then $M=\bigcup_{R=1}^\infty \overline B_R(x_0)$.
  For $R\in \mathbb N$, consider also the subcategory $\hat \categ_R$ whose objects are those $\nu\in\categ_{\mathrm{AS}}$ supported on the compact subset of $\mathcal X$ consisting of points $(x,v_1,\dots,v_n,\mathbbm t)$ with $x\in \overline  B_R(x_0)$, $\|v_1\|,\dots,\|v_n\|\leq R$, and $\mathbbm t\in \dot T^*M$ (observe that $\dot T^*M$ is topologically a sphere, so it is compact), with morphisms corresponding to the curves of those argute slices contained in  $\categ_R$. It follows from the same argument as in Lemma \ref{lem:sigmacompact} that $\categ_R$ is sequentially compact. Since $\categ_{\mathrm{AS}}=\bigcup_R\categ_R$, it is $\sigma$-compact. 
  
  On the other hand,  $\tilde\categ_{\mathrm{AS}}$ is a closed subset of $\categ_{\mathrm{AS}}$ as a consequence of its maximality and the sequential compactness: if $(\nu_1^i)_i$ and $(\nu_2^i)_i$ are two converging sequences of argute slices in $\tilde\categ_{\mathrm{AS}}$, $\nu^i_1\to\nu^0_1$ and $\nu^i_2\to\nu^0_2$, and if $\mu^i\in\Hom^t(\nu_1^i,\nu_2^i)$, then we can use an argument analogous to the proof of Lemma \ref{lem:sigmacompact} to realize all the $\mu^1$ as curves in $\categ_N$ for $N$ large, and then we can apply the Prokhorov theorem \cite{prokhorov} to obtain a converging subsequence $\mu^{i_j}\to \mu_0\in\Hom^t(\nu^0_1,\nu^0_2)$, so $\nu_1^0,\nu^0_2\in\tilde\categ_{\mathrm{AS}}$.
  
  Properties \ref{it:lipschitzity}--\ref{it:sufficientlymanyshortcurves} follow from the arguments analogous to those used in the proof of Lemma \ref{lem:lipschitzity}.
 \end{proof}

 \begin{lem}\label{lem:kamamenable}
  Assume \ref{Lfirst}--\ref{Llast} hold. Then the category $\categ_{\mathrm{AS}}$ is Lagrangian and finely \kam{} amenable.
 \end{lem}
 \begin{proof}[Proof of Lemma \ref{lem:kamamenable}]
  First, we note that it is clear that \ref{A:curves} and \ref{A:division} hold. By Lemma \ref{lem:sufflagrangian}, $\tilde\categ_{\mathrm{AS}}$ is Lagrangian.

  In the notation of Section \ref{sec:weakkamsolutions}, for $\eta,\nu\in\tilde\categ_{\mathrm{AS}}$ and $\gamma\in\Hom(\eta, \nu)$ given by a curve $\gamma$ of argute slices, $\gamma=(\nu_t)_{t\in I}$, we let $\bar\gamma(t)=\nu_t$. With this definition, it is immediate that \ref{A:lagrangian}--\ref{A:restriction} hold.

  Let $\nu\in \tilde\categ_{\mathrm{AS}}$, and $u$ a function satisfying \eqref{eq:fixedpoint}, and let us show that \ref{A:minimum} holds; we will do so with $k_0=1$. 
  Let 
  \[m=\inf_{\substack{\eta\in \tilde\categ_{\mathrm{AS}}\\\gamma\in \Hom(\eta,\nu)\\\mathbf t(\gamma)=2}}u(\eta)+A(\gamma).\]

  By \eqref{eq:fixedpoint}, we know that $u$ belongs to $\mathcal H(c_0)$ for some $c_0$, and by Lemma \ref{lem:auxlemmaLO}\ref{it:lipschitzityofHc}, we know that $u$ is Lipschitz continuous, so that there is some constant $C>0$ such that, in particular,
  \[|u(\eta)-u(\nu)|\leq C\dist(\nu,\eta).\]
  Let $k_1=C+1$ and let $k_2$ be as in \ref{L2}, so that we have
  \[(C+1)(\dist_M(x_0,x)+\|v_1\|+\dots+\|v_n\|)-k_2\leq L(x,v_1,\dots,v_n),\qquad (x,v_1,\dots,v_n)\in T^nM.\]
  It follows that, if the pair $(\eta,\gamma)$ contending in the definition of $m$ satisfies $u(\eta)+A(\gamma)\leq m+\epsilon$, then
  \begin{align*}
      m+\epsilon
      &\geq u(\eta)+\int_{\mathcal X}L(x,v_1,\dots,v_n)d\gamma(x,v_1,\dots,v_n,\mathbbm t)\\
      &\geq u(\eta)+\int_{\mathcal X}((C+1)(\dist_M(x_0,x)+\|v_1\|+\dots+\|v_n\|)-k_2)\,d\gamma(x,v_1,\dots,v_n,\mathbbm t)\\
      &\geq \left(u(\nu)-C\dist(\nu,\eta)\right)+(C+1)\int_{\mathcal X}\dist_M(x_0,x)+\|v_1\|+\dots+\|v_n\|d\gamma-k_2\mathbf t(\gamma) \\
      &\geq u(\nu)-C\int_{\mathcal X}\dist_M(x_0,x)+\|v_1\|+\dots+\|v_n\|d\gamma \\
      &\qquad \qquad\qquad+(C+1)\int_{\mathcal X}\dist_M(x_0,x)+\|v_1\|+\dots+\|v_n\|d\gamma-2k_2 \\
      &=u(\nu)+\int_{\mathcal X}\dist_M(x_0,x)+\|v_1\|+\dots+\|v_n\|d\gamma-2k_2.
  \end{align*}
  In other words, 
  \[\int_{\mathcal X}\dist_M(x_0,x)+\|v_1\|+\dots+\|v_n\|d\gamma(x,v_1,\dots,v_n,\mathbbm t)\leq m+\epsilon-u(\nu)+2k_2.\]
  Observe that this means, in particular, that $m$ is finite.

  We now use this estimate to show that the set of curves we are interested in is sequentially compact.
  Let, for $r>0$, $K_r\subset \mathcal X$ be the set of points $(t,x,v_1,\dots,v_n,\mathbbm t)$ with $0\leq t\leq 2$, $\dist_M(x_0,x)+\|v_1\|+\dots+\|v_n\|\leq r$, and $\mathbbm t\in \dot T_x^*M$. 
  From the definition of $K_r$, it is clear that the set $\pi^1(K_r)$ is closed and bounded, so it is compact, and that the set $\pi^2(K_r)$ is closed. The set $\pi^2(K_r)$ is contained in the compact set $\bigcup_{\dist_M(x_0,x)\leq r}\dot T^*_xM$, so it is compact as well.
  This means that the set $[0,2]\times\pi^1(K_r)\times \pi^2(K_r)$  is compact, and so is its closed subset $K_r$. 
  Applying Markov's inequality, we have
  \begin{multline*}
      \gamma(\mathcal X\setminus K_r)
      \leq \frac1r\int_{\mathcal X}\dist_M(x_0,x)+\|v_1\|+\dots+\|v_n\|\,d\gamma(x,v_1,\dots,v_n,\mathbbm t)\\
      \leq \frac{m+\epsilon-u(\nu)+2k_2}r.
  \end{multline*}
  Let $\mathscr C$ be the set of curves $\gamma$ appearing in the pairs $(\eta,\gamma)$ contending in the definition of $m$ and with $u(\eta)+A(\gamma)\leq m+\epsilon$.
  The curves in $\mathscr C$ are contained in the set $Y$ of families of positive Radon measures on $\mathcal X$ with total mass $2$ and whose mass on the complement $\mathcal X\setminus K_r$ is uniformly at most $\frac{m+\epsilon-u(\nu)+2k_2}{r}$, which is to say that this collection of measures is tight. By the Prokhorov theorem, $Y$ is sequentially compact. 
  Since $\mathscr C$ is closed and contained in $Y$, $\mathscr C$ is sequentially compact as well. 
  
  Since $m$ is finite, we can take sequences $(\eta_i)$ and $(\gamma_i)$, with $\eta_i\in \tilde\categ_{\mathrm{AS}}$ and $\gamma_i\in \Homt{2}(\eta_i,\nu)$, such that $u(\eta_i)+A(\gamma_i)\to m$. Since $\gamma_i\in \mathscr C$, we can assume, perhaps after passing to a subsequence, that $(\gamma_i)$ converges to some curve $\hat\gamma$. This means that also the sequences $A(\gamma_i)$ and $u(\eta_i)$ converge. Observe that, since $u$ satisfies \eqref{eq:fixedpoint}, we have
  \[ u(\nu)=\lim_{i\to+\infty}u(\eta_i)+A(\gamma_i)+c_0\mathbf t(\gamma_i)=m+2c_0.\]

  Disintegrate $\hat\gamma$ over the projection $(t,x,v_1,\dots,v_n,\mathbbm t)\mapsto t\in[0,2]$ to get, for almost every $t\in[0,2]$, a Borel probability measure $\vartheta_t$ on $\mathcal X$; it is routine to check that, for almost every $t$, $\vartheta_t$ is an argute slice. Fix $0< t_0< 1$ such that $\vartheta_{t_0}\eqqcolon\vartheta_0$ is an argute slice.

  We claim that the pair $(\vartheta_0,\hat\gamma|_{[t_0,2]})$ achieves the minimum in \ref{A:minimum} with $t=2-t_0>k_0=1$.  
  Let us explain why this is true.
  Since $u$ satisfies \eqref{eq:fixedpoint}, it also satisfies \eqref{eq:dominated},
  which gives a lower bound for the minimum in \ref{A:minimum}: we have
  \begin{multline*}
      u(\vartheta_0)+A(\hat\gamma|_{[t_0,2]})\geq
      \inf_{\substack{\eta\in \tilde\categ_{\mathrm{AS}}\\\gamma\in\Hom(\eta,\nu)\\\mathbf t(\gamma)=2-t_0}}u(\eta)+A(\gamma)\\
      \geq u(\nu)-c_0(2-t_0)=m+2c_0-c_0(2-t_0)=m+c_0t_0.
  \end{multline*}
  On the other hand, from our choice of $\hat\gamma$ and $\vartheta_0$ and again using \eqref{eq:dominated}, we have
  \begin{multline*}
      u(\vartheta_0)+A(\hat\gamma|_{[t_0,2]})
      =\lim_{i\to+\infty} u(\vartheta_0)+A(\gamma_i|_{[t_0,2]})\\
      \leq \lim_{i\to+\infty} A(\gamma_i|_{[0,t_0]})+u(\eta_i)+c_0t_0+A(\gamma_i|_{[t_0,2]})\\
      =\lim_{i\to+\infty} A(\gamma_i|_{[0,2]})+u(\eta_i)+c_0t_0
      =m+c_0t_0.
  \end{multline*}
  This means that $u(\vartheta_0)+A(\hat\gamma|_{[t_0,2]})$ is as small as the lower bound we had for the infimum that interests us, so it achieves the minimum. This shows that \ref{A:minimum} holds in $\tilde\categ_{\mathrm{AS}}$, and hence finishes the proof that this category is finely \kam{} amenable.
 \end{proof}

\section{Characterization of minimizable La\-gran\-gians}
\label{sec:coarsechar}
In this section we characterize the Lagrangian actions minimizable by holonomic measures in Theorem \ref{thm:coarse}. We then connect the weak \textsc{kam} solutions from Theorem \ref{thm:weakkamsol} with exact differential forms in Corollary \ref{cor:exactdescent}.

We will work in the context of the Lagrangian category $\categ_{\mathrm{HM}}$ defined in Section \ref{sec:skeleton}, so that we are given a manifold $M$ and a Borel-measurable function $L\colon T^nM\to\R$. The action $A$ is given by integration of $L$ with respect to the holonomic measures that constitute the morphisms of $\categ_{\mathrm{HM}}$.

We denote by $C^\infty(T^nM)$ the space of infinitely-differentiable functions on the bundle $T^nM$, and we let $\mathscr E'(T^nM)$ denote the space of compactly-supported distributions on $T^nM$, which is dual to $C^\infty(T^nM)$. The set $\mathscr E'(T^nM)$ contains, in particular, all compactly-supported, Radon measures $\mu$ on $T^nM$. 

Fix two objects $T_1,T_2$ in $\categ_{\mathrm{HM}}$ such that the set $\Hom(T_1,T_2)$ is not empty. In other words, $T_1$ and $T_2$ are normal $(n-1)$-dimensional currents without boundary, and there exist compactly-supported Radon measures $\mu$ on $T^nM$ such that the induced current $T_\mu$ has boundary $\partial T_{\mu}=T_2-T_1$, so that $\mu\in\Hom(T_1,T_2)$.

Recall a topological vector space $V$ is \emph{sequential} if for every set $S\subset V$ and every element $s$ in the closure $\overline S$ there exists a sequence of points $\{s_i\}_i\subseteq S$ such that $s_i\to s$. This is verified if $V$ is normed, metric, or first countable (that is, if every point has a countable neighborhood basis).

Let $E$ be a complete, sequential, locally-convex topological vector space of Borel measurable functions on $T^nM$ that contains $C^\infty(T^nM)$ as a subspace. Assume that every element of $\Hom(T_1,T_2)$ induces a continuous linear functional on $E$, i.e., $\Hom(T_1,T_2)\subseteq E^*$, and that the topology of $C^\infty(T^nM)$ is finer than the one this subspace inherits from $E$, or, in other words, that every open set in the inherited topology is an open set in the topology induced by the seminorms
\[|f|_{K,k}=\sum_{|I|\leq k}\sup_{x\in K}\left|\partial^If(x)\right|,\]
where $k\geq 0$, $K\subset T^nM$ is compact, and the sum is taken over all multi-indices $I=(i_1,\dots,i_n)$ with $|I|=\sum_ri_r\leq k$. This assumption implies that every continuous linear functional $\vartheta\in E^*$ defines a compactly-supported distribution when restricted to $C^\infty(T^nM)$. For example, $E$ can be the space $C^k(T^nM)$, $k\in[0,+\infty]$, with the topology of uniform convergence on compact sets of the derivatives of order $\leq k$.

\begin{thm}\label{thm:coarse}
 If $L$ is an element of $E$ such that its action functional $A$ reaches its minimum within $\Hom(T_1,T_2)$ at some measure $\mu$, then there exist differential forms $\omega_1,\omega_2,\dots$ in $\Omega^{n-1}(M)$, and nonnegative functions $g_1,g_2,\dots$ in $E$ such that
 \[\lim_{i\to+\infty}\int_{T^nM}g_i\,d\mu=0,\]
 and
 \[L=\lim_{i\to+\infty}d\omega_i+g_i,
 \]
 where the limit is taken in $E$. In particular, 
 \[\int_{T^nM} L\,d\mu=\lim_i\partial T_{\mu}(\omega_i)=T_2(\omega_i)-T_1(\omega_i).\]
\end{thm}
\begin{rmk}\label{rmk:hamiltonjacobi}
 In many cases, like in the situation of Corollary \ref{cor:exactdescent} below, using the version of the Arzel\`a--Ascoli theorem given in Lemma \ref{lem:arzelaascoli}, one can extract a Lipschitz limit $\omega$ of the forms $\omega_i$. It then satisfies 
 \[d\omega -L\leq 0,\]
 with equality $\mu$-almost everywhere. Here one can think of \[H(x,v,p)=p-L(x,v)\] as the Hamiltonian, so that we are looking at a sort of \emph{Hamilton-Jacobi equation}, \[H(x,v,d\omega)=0.\] In this sense, we can say that $\omega$ is a \emph{critical subsolution of the Hamilton-Jacobi equation}; cf. \cite{fathibook}. 
\end{rmk}

In order to prove the theorem, we need

\begin{lem}\label{lem:dualsets}
 In the setting of Theorem \ref{thm:coarse}, let
 \begin{align*}
  Q=&\{\ell\colon E^*\to\R\;|\;\textrm{$\ell$ is affine and continuous, $\ell(\mu)\geq 0$ for all $\mu\in\Hom(T_1,T_2)$}\},\\
  R=&\{\ell\colon E^*\to\R\;|\;\textrm{$\ell(\xi)=\xi(d\omega+g)- T_2(\omega)-T_1(\omega)$,}
   \textrm{  $\omega\in \Omega^{n-1}(M)$, $g\in E$, $g\geq 0$}\}.
 \end{align*}
 Then we have $Q=\overline R$ in $E$.
\end{lem}

\begin{proof}[Proof of Lemma \ref{lem:dualsets}]
 For a convex subset $A$ of a topological vector space, we will denote by $A'$ the set of real-valued continuous affine functionals in that are nonnegative on $A$. 
 
 We first observe that $Q'=\Hom(T_1,T_2)=R'$. To see why, note that the set of functionals induced by nonnegative elements of $C^\infty(T^nM)$ is a subset both of $Q$ and of $R$ (in the latter case, take $\omega=0$ and $g\in C^\infty(T^nM)\subseteq E$), so by \cite[\S6.22]{liebloss} all elements of $Q'$ and $R'$ can be represented as integration over compactly-supported, nonnegative, Radon measures. Also, if $\omega\in\Omega^{n-1}(M)$, then the affine functional 
 \[\ell_\omega(\xi)=\xi(d\omega)-T_2(\omega)+T_1(\omega)\]  
 belongs to both $Q$ and $R$, so it is nonnegative throughout $Q'$ and $R'$. Since its negative $\ell_{-\omega}=-\ell_\omega$ is, for the same reason, nonnegative throughout $Q'$ and $R'$, we conclude that 
 \[0=\ell_\omega(\xi)=\xi(d\omega)-T_2(\omega)+T_1(\omega)=\partial T_\xi(\omega)-T_2(\omega)+T_1(\omega)\] 
 for all $\xi\in Q'\cup R'$ and all $\omega\in \Omega^{n-1}(M)$. In other words, the current $T_\xi$ induced by the measure  representing $\xi$ has boundary $\partial T_\xi=T_2-T_1$. So indeed $Q'=\Hom(T_1,T_2)=R'$.

 Since $Q'=R'$, we have $\overline Q=\overline R$. 
 Indeed, if there were some $q\in \overline Q\setminus \overline R$, then the Hahn--Banach separation theorem would produce a continuous affine functional $\ell\colon E\to\R$ with $\ell(r)<a<b<\ell(q)$ for some $a,b\in\R$ and all $r\in \overline R$, whence the continuous affine functional $\ell-a$ would be positive on $R$ and not on $Q$, contradicting $Q'=R'$; a similar situation would arise if  $\overline R \setminus \overline Q\neq \emptyset$.

 The claim of the lemma follows from the fact that $Q$ is closed.
\end{proof}

\begin{proof}[Proof of Theorem \ref{thm:coarse}]
 The functional $\ell(\xi)=\int L\,d\xi-\int L\,d\mu$ belongs to the set $Q$ in the statement of Lemma \ref{lem:dualsets}, and by the lemma it also belongs to $\overline R$. The sequentiality of $E$ implies that the topological closure equals the sequential closure, so there exists a sequence of affine functions $\ell_i\in R$ of the form
 \[\ell_i(\xi)=\xi(g_i+d\omega_i)-T_2(\omega_i)-T_1(\omega_i),\quad i=1,2,\dots,\] 
 with $g_i\in E$, $\omega_i\in \Omega^{n-1}(M)$, $ g_i\geq 0$,and converging to $\ell_i\to \ell$. 
 Comparing the linear and constant parts of the functionals $\ell$ and $\ell_i$, we conclude that $\int L\,d\mu=\lim_i T_2(\omega_i)-T_1(\omega_i)$.
 We also have that 
 \[0=\ell(\mu)=\lim_i\ell_i(\mu)=\lim_i\int_{T^nM} g_i\,d\mu,\] 
 where the last equality is true because the boundary of $T_\mu$ is $T_2-T_1$, that is $\int d\omega_i\,d\mu= T_2(\omega)-T_1(\omega_i)$. 
\end{proof}
\begin{coro}\label{cor:exactdescent}
 Let $L\colon T^nM\to\R$ be a $C^0$ function satisfying the hypotheses of Theorem \ref{thm:weakkamsol}.
 Assume that there is a family $\{\gamma_\alpha\}_{\alpha\in I}$ of curves of argute slices %
 such that the support of the holonomic measures $\mu_{\gamma_\alpha}$ covers almost all of $M$, that is, such that if $\gamma_\alpha=(\nu_s^\alpha)_{s\in[0,T_\alpha]}$ for $\alpha\in I$, then the complement of the set
 \[
  \bigcup_{\alpha\in I} \pi_M(\supp\mu_{\gamma_\alpha})
  =\bigcup_{\alpha\in I}\overline{\bigcup_{s\in[0,T_\alpha]} \pi_M(\supp\nu_s^\alpha)}
  \subseteq M
 \]
 has Lebesgue measure zero on $M$. 
 
 Moreover, assume that the curves $\gamma_\alpha$ minimize the action $A$ simultaneously, in the sense that every convex combination of the associated holonomic measures $\mu_{\gamma_\alpha}$ minimizes $A$ with respect to all positive, compactly-supported, Radon measures with the same boundary.
 
 Then the function $u$ in the conclusion of Theorem \ref{thm:weakkamsol} corresponds to a Lipschitz form $\omega\in \Omega^{n-1}(M)$,  and for all $\alpha\in I$, $t\in[0,T_\alpha]$ we have
 \[ T_{\nu^\alpha_t}(\omega)=u(\nu^\alpha_t).\]
\end{coro}

\begin{proof}
 Form a convex combination of all the the associated holonomic measures $\mu_{\gamma_\alpha}$ using a probability measure supported throughout $I$. Apply Theorem \ref{thm:coarse} to obtain forms $\{\omega_j\}_{j=1}^\infty\subset\Omega^{n-1}$ as in the statement of that result, and then apply also Lemma \ref{lem:arzelaascoli} to obtain a subsequence of the forms $\omega_j$ converging to a Lipschitz differential form $\omega$ that corresponds to $u$ as statement of the corollary.
\end{proof}

\begin{lem}[Arzel\`a--Ascoli for sections of a vector bundle]\label{lem:arzelaascoli}
 Let $\alpha_1,\alpha_2,\dots$ be a sequence of smooth, uniformly bounded as sections of a vector bundle $F\to M$ on the compact manifold $M$, and assume that their derivatives $d\alpha_1,d\alpha_2,\dots$ are also uniformly bounded. Then there is a subsequence $\{i_j\}_{j=1}^\infty\subset\N$ such that the sequence $\alpha_{i_1},\alpha_{i_2},\dots$ converges uniformly to a Lipschitz section $\alpha$ of $F$. 
\end{lem}
\begin{proof}
 Take a finite set of sections $\beta_1,\beta_2,\dots,\beta_N$ of $F$ such that, for each $x\in M$, the set $\{\beta_1(x),\dots,\beta_N(x)\}$ is a basis for the fiber of $F$ at $x$. Express each $\alpha_i$ as $\alpha_i(x)=\sum_{j=1}^N\phi_i^j(x)\beta_N$, for some smooth functions $\phi_i^j$. Note that the uniform boundedness of $d\alpha_i$ implies the uniform boundedness of the derivatives $d\phi_i^j$. Apply the classical Arzel\`a--Ascoli result for Lipschitz functions to the sequences $\{\phi_i^j\}_{i=1}^\infty$, $j=1,\dots,N$, successively so as to obtain a subsequence for which all $N$ sequences converge simultaneously to Lipschitz functions $\phi^j$, which gives the statement of the lemma with $\alpha(x)=\sum_{j=1}^N\phi^j(x)\beta_j$.
\end{proof} 

 \bibliography{bib}{}
 \bibliographystyle{plain}
\end{document}